\documentclass[12pt, leqno]{article}

\title{
Voros Coefficients for the Hypergeometric Differential Equations 
and Eynard-Orantin's Topological Recursion  \\
{\Large --- Part I : For the Weber Equation ---}
}
\author{
Kohei \textsc{Iwaki}$^{\dagger}$\and
Tatsuya \textsc{Koike}$^{\sharp}$\and
Yumiko \textsc{Takei}$^{\heartsuit}$}
\def\paperinfo{
\renewcommand{\thefootnote}{\fnsymbol{footnote}}
\footnote[0]{$^{\dagger}$Graduate School of Mathematics, 
Nagoya University. 
}
\footnote[0]{$^{\sharp}$Department of Mathematics, 
Graduate School of Science, Kobe University.}
\footnote[0]{$^{\heartsuit}$Department of Mathematics, 
Graduate School of Science, Kobe University. 
}
\footnote[0]{2010 \textit{Mathematics Subject Classification}. 
Primary:34M60; Secondary:81T45}
\footnote[0]{\textit{Keywords}: Exact WKB analysis; Voros coefficients;
Topological recursion; Quantum curves; Free energy; Weber equaiton.}
\renewcommand*{\thefootnote}{\arabic{footnote}}
}
\date{\today}

\usepackage{afterpage}
\usepackage{amsmath}
\usepackage{amsthm}
\usepackage{amssymb}        
\usepackage{amsfonts}       
\usepackage[dvipdfmx]{color}
\usepackage[dvipdfmx]{graphicx}
\usepackage{enumitem}
\usepackage[mathscr]{eucal} 
\usepackage[a4paper]{geometry}
\usepackage{lipsum}
\usepackage{lscape}
\usepackage{mathrsfs}
\usepackage{stmaryrd}       
\usepackage[all]{xy}

\theoremstyle{plain}
\newtheorem{thm}{Theorem}[section]
\newtheorem{prop}[thm]{Proposition}
\newtheorem{lem}[thm]{Lemma}

\theoremstyle{definition}
\newtheorem{dfn}[thm]{Definition}
\theoremstyle{remark}
\newtheorem{rem}[thm]{Remark}

\numberwithin{equation}{section}
\numberwithin{table}{section}
\numberwithin{figure}{section}

\newcommand{\thmref}[1]{Theorem \ref{thm:#1}}

\def\Res{\mathop{\rm{Res}}}
\def\Im{{\mathop{\rm{Im}}\nolimits}}
\def\ord{\mathop{\rm{ord}}\nolimits}
\def\Disc{\mathop{\rm{Disc}}\nolimits}
\def\Sing{\mathop{\rm{Sing}}\nolimits}

\begin{document}

\maketitle

\paperinfo

\vspace{-1.em}
\begin{abstract}
We develop the theory of quantization of spectral curves 
via the topological recursion. We formulate a quantization 
scheme of spectral curves which is not necessarily admissible
in the sense of Bouchard and Eynard. 
The main result of this paper and the second part \cite{IKT-part2}
establishes a relation between 
the Voros coefficients for the quantum curves 
and the free energy for spectral curves 
associated with the confluent family 
of Gauss hypergeometric differential equations.
We focus on the Weber equation in this article, 
and generalize the result for the other members of the confluent family
in the second part.  We also find explicit formulas of free energy 
for those spectral curves in terms of the Bernoulli numbers. 
\end{abstract}

\tableofcontents


\section{Introduction} 
\label{section:intro}
The Voros coefficient for a (1-dimensional) Schr\"odinger-type 
linear ordinary differential equation is defined as a contour integral 
of the logarithmic derivative of WKB solutions. 
Its importance in the study of the global behavior of solutions of 
differential equations has been already recognized 
by the earlier work of Voros (\cite{Voros83}). 
For example, the Voros coefficient is an important ingredient 
for describing the Stokes phenomena and the monodromy group 
of a Schr\"odinger equation. 
Moreover, concrete form of the Voros coefficient enables us 
to analyze the parametric Stokes phenomena explicitly. 
The concrete form of the Voros coefficient is now known 
for the Weber equation, the Whittaker equation, 
the Kummer equation and the Gauss hypergeometric equation 
(\cite{SS, Takei08, KoT11, ATT, Aoki-Tanda, AIT}). 
The Voros coefficient also plays an essentially important role 
in a relationship between the exact WKB analysis and cluster algebras (\cite{IN14}).

On the other hand, the topological recursion introduced by 
B.~Eynard and N.~Orantin (\cite{EO}) is a generalization of the loop equations that 
the correlation functions of the matrix model satisfy. 
For a Riemann surface $\Sigma$ and meromorphic functions 
$x$ and $y$ on $\Sigma$, it produces an infinite tower 
of meromorphic differentials $W_{g,n}(z_1, \ldots, z_n)$ on $\Sigma$. 
A triplet $(\Sigma, x, y)$ is called a spectral curve and 
$W_{g,n}(z_1, \ldots, z_n)$ is called a correlation function.
Moreover, for a spectral curve, we can define free energies 
(also called symplectic invariants) $F_g$. 
Topological recursion attracts the interests of both mathematicians 
and physicists since the quantities produced by its framework 
are expected to encode information of various geometric or enumerative invariants. 
It is also known that the topological recursion is closely related 
to integrable systems. For more details see, e.g., the review paper \cite{EO-08}. 

A surprising connection between WKB theory and topological 
recursion was discovered recently. 
The quantization scheme connects WKB solutions with 
the topological recursion (\cite{GS, DM, BE} etc.). 
More precisely, it was found that WKB solutions can be constructed 
by correlation functions for the spectral curve, which corresponds 
to the classical limit of the differential equation, 
when the spectral curve satisfies the ``admissibility condition" 
in the sense of \cite[Definition 2.7]{BE}. 

Then the following question naturally arises: 
\begin{quote}
What quantity corresponds to the Voros coefficient in the topological recursion ?
\end{quote}
In this paper and the second paper \cite{IKT-part2}, 
we answer this question for the confluent family of 
the Gauss hypergeometric differential equations. 
That is, we show that the Voros coefficients are expressed 
as the difference values of the free energy of the spectral curve 
obtained as the classical limit of the family of hypergeometric equations. 
This is our first main result. 

As a model example, in this first paper 
we will consider the Weber equation
\begin{equation} \label{eq:Weber-intro}
\left\{ \hbar^2 \frac{d^2}{dx^2} - 
\left( 
	\frac{x^2}{4} - \lambda + \frac{\hbar \nu}{2}
\right) 
\right\} \psi = 0
\end{equation}
and the corresponding spectral curve 
\begin{equation} \label{eq:Weber-curve-intro}
y^2 - \left( \frac{x^2}{4} - \lambda \right) = 0
\end{equation} 
obtained as the classical limit of \eqref{eq:Weber-intro}.
Here $\lambda$ is a non-zero parameter constant which is related 
to the formal monodromy of \eqref{eq:Weber-intro} around $x=\infty$. 
The Weber equation \eqref{eq:Weber-intro} is the quantum curve 
of the Weber curve \eqref{eq:Weber-curve-intro}, as is shown 
in \cite{BE} (see also \thmref{WKB-Wg,n}).

Let $V(\lambda,\nu;\hbar) = \sum_{m = 1}^{\infty} \hbar^{m} V_m(\lambda,\nu)$ 
be the Voros coefficient of \eqref{eq:Weber-intro}
(see \S \ref{subsec:Voros-coeff} and \S \ref{subsec:Weber-Voros} for precise definition), 
and $F(\lambda;\hbar) = \sum_{g = 0}^{\infty} \hbar^{2g - 2} F_g(\lambda)$ 
be the free energy of the Weber curve 
(see \S \ref{subsec:free-energy} for precise definition). 
Then, one of our main result is formulated as follows:

\begin{thm}[cf. \thmref{weber:main(i)}]
The Voros coefficient of the Weber equation \eqref{eq:Weber-intro} 
and the free energy of the Weber curve \eqref{eq:Weber-curve-intro}
are related as follows:
\begin{equation}
V(\lambda, \nu; \hbar)
= 
F \left({\lambda} - \frac{\hbar \nu}{2} +\frac{\hbar}{2}; \hbar \right)
- F \left({\lambda} - \frac{\hbar \nu}{2} - \frac{\hbar}{2} ; \hbar \right)
- \frac{1}{\hbar} \frac{\partial F_0}{\partial \lambda} (\lambda) 
+ \frac{ \nu }{2} \frac{\partial^2 F_0}{\partial \lambda^2}(\lambda). 
\end{equation}
\end{thm}

The above formula establishes a relation between 
the Voros coefficient and the free energy. 
We will generalize the result for the confluent family 
of the hypergeometric differential equations and 
associated spectral curves in \cite{IKT-part2}. 

To obtain the above result and the main results of \cite{IKT-part2}, 
we also study quantizations without assuming the 
``admissibility condition" of \cite{BE} 
in the case when the degree of a polynomial $P(x,y)$ with respect to $y$ is two. 
This is done in \thmref{WKB-Wg,n}, which is a partial extension of that of \cite{BE}. 
We need this generalization because the spectral curve \eqref{eq:Gauss-curve-intro}
arising from the Gauss hypergeometric differential equation 
(which will be discussed in \cite{IKT-part2} in detail) 
is not admissible. 

As applications of the main results, we get 
three-term difference equations 
which the free energy satisfies (Theorem \ref{thm:weber:main(ii)}). 
By solving them, we obtain concrete forms of the free energy and
the Voros coefficient as well 
(Theorem \ref{cor:weber:FreeEnergy} and Theorem \ref{cor:weber:voros}). 
For example, the explicit form of the $g$-th free energy $F_g$ of the 
Weber curve \eqref{eq:Weber-curve-intro} is given by 
\begin{equation}
F_{g} = \frac{B_{2g}}{2g(2g-2)} \, \frac{1}{\lambda^{2g-2}} 
\quad (g \ge 2),
\end{equation}
which recovers the known result by Harer-Zagier \cite{HZ}
on the computation of the Euler characteristic of 
the moduli space of Riemann surfaces with genus $g$. 
Throughout the paper, $B_{m}$ is the $m$-th Bernoulli number 
(see \eqref{def:Bernoulli} for the definition).

Our method works perfectly for spectral curves 
arising from the confluent family of Gauss hypergeometric equations.
For example, the spectral curve 
\begin{equation} \label{eq:Gauss-curve-intro}
y^2 - \frac{ {\lambda_{\infty}}^2 x^2 - 
({\lambda_{\infty}}^2 + {\lambda_0}^2 - {\lambda_1}^2)x + {\lambda_0}^2 }{x^2 (1 - x)^2} = 0
\end{equation}
of the Gauss hypergeometric equation gives 
\begin{align}
F^{\rm Gauss}_{g}
&= \frac{B_{2g}}{2g(2g-2)}
\left\{
\frac{1}{(\lambda_0 + \lambda_1 + \lambda_{\infty})^{2g-2}}
+ \frac{1}{(\lambda_0 - \lambda_1 + \lambda_{\infty})^{2g-2}}
\right.
\\
&\quad\quad
\left.
+ \frac{1}{(\lambda_0 + \lambda_1 - \lambda_{\infty})^{2g-2}}
+ \frac{1}{(\lambda_0 - \lambda_1 - \lambda_{\infty})^{2g-2}}
- \frac{1}{(2\lambda_0)^{2g-2}}
- \frac{1}{(2\lambda_1)^{2g-2}}
- \frac{1}{(2\lambda_{\infty})^{2g-2}}
\right\} \notag
\end{align}
as the $g$-th free energy for $g \ge 2$ 
(for generic $\lambda_0, \lambda_1, \lambda_\infty$). 
The free energies of other examples which will be considered 
in \cite{IKT-part2} are also expressed in terms of the Bernoulli numbers. 
Consequently, the Voros coefficients of the quantum curves 
are written in terms of the Bernoulli polynomials, 
and these formulas recover the results in 
\cite{SS, Takei08, KoT11, ATT, Aoki-Tanda, AIT}.
We will show these results in our second paper \cite{IKT-part2}, 
based on the results 
(mainly \thmref{WKB-Wg,n} on the quantization of the spectral curve 
given in \S \ref{sec:quantization-spectral-curve}) 
of this paper.

The paper is organized as follows: 
In \S2 we recall some fundamental facts about 
the exact WKB analysis and Eynard-Orantin's topological recursion. 
In \S3 we study quantization of the spectral curve. 
Our main result in this section is \thmref{WKB-Wg,n}, 
which gives quantizations without assuming 
the admissibility condition in the case when 
the spectral curve is described as a polynomial equation 
$P(x,y)=0$ which is degree 2 in $y$. 
In \S4 we state our main theorem for the Weber equation and the Weber curve. 
Moreover, as an application of the results we give concrete 
forms of the free energy and the Voros coefficient. 

\section*{Acknowledgement}
We are grateful to 
Takashi Aoki, 
Takahiro Kawai,
Toshinori Takahashi,
Yoshitsugu Takei 
and 
Mika Tanda
for helpful discussions and communications.
This work is supported, in part, by JSPS KAKENHI Grand Numbers 
16K17613, 16H06337, 16K05177, 17H06127.

\section{Voros coefficients and topolotical recursion}
\label{sec:review}

In this section we briefly recall some basics about
the exact WKB analysis, and Eynard-Orantin's theory.
See \cite{KT98} and \cite{EO} (or \cite{EO-08})
respectively for the details of them.

\subsection{WKB solution}
\label{sec:WKB-and-Voros}

The differential equation which we will discuss in this paper
is the second order ordinary differential equation with a small 
parameter $\hbar \ne 0$ of the form
\begin{equation}
\label{eq:2nd-ODE}
\left\{
\hbar^2 \frac{d^2}{dx^2} + q(x, \hbar) \hbar \frac{d}{dx} + r(x, \hbar)
\right\}\psi = 0,
\end{equation}
where $x \in \mathbb{C}$, and
\begin{equation}
q(x, \hbar) = q_{0}(x) + \hbar q_{1}(x),
\quad
r(x, \hbar) = r_{0}(x) + \hbar r_{1}(x) + \hbar^2 r_2(x),
\end{equation}
with rational functions $q_{j}(x)$ and $r_{j}(x)$
($j = 0, 1, 2$).
We consider \eqref{eq:2nd-ODE} as a differential equations
on the Riemann sphere $\mathbb{P}^1$ with 
regular or irregular singular points.
A WKB type solution of \eqref{eq:2nd-ODE}
is a (formal) solution of \eqref{eq:2nd-ODE} of the form
\begin{equation}
\label{eq:WKB-type}
\psi (x, \hbar)
= \exp \left[
\frac{1}{\hbar} f_{-1}(x)
+ f_0(x) +
\hbar f_1(x) + \cdots\right].
\end{equation}
A typical way of constructing WKB type solutions is the following:
The logarithmic derivative $S(x, \hbar)$ of solutions of \eqref{eq:2nd-ODE}
satisfies the Riccati equation
\begin{equation}
\label{eq:Riccati-gen}
\hbar^2 \left( \frac{d}{dx} S(x, \hbar) + {S(x, \hbar)}^2 \right) 
+ \hbar q(x, \hbar) S(x, \hbar) + r(x, \hbar) = 0.
\end{equation}
Eq. \eqref{eq:Riccati-gen} admits a solution of the form
\begin{equation}
\label{eq:Riccati-gen-expansion}
S(x, \hbar) := \hbar^{-1} S_{-1}(x) + S_0(x) + \hbar S_1(x) + \cdots 
= \sum_{m = -1}^{\infty} \hbar^m S_m(x).
\end{equation}
In fact, by substituting \eqref{eq:Riccati-gen-expansion} into
\eqref{eq:Riccati-gen}, and
equating the both-sides like powers with respect to $\hbar$,
we obtain
\begin{align}
\label{eq:Riccati-gen-1}
S_{-1}^2 + q_{0}(x) S_{-1} + r_{0}(x)  &= 0,\\
\label{eq:Riccati-gen-2}
\big(2 S_{-1} + q_{0} (x)\big) S_0 + q_{1}(x) S_{-1} + r_{1}(x)
+ \frac{d S_{-1}}{dx}  &= 0, \\
\label{eq:Riccati-gen-2-2}
\big(2 S_{-1} + q_{0} (x)\big) S_1 + S_{0}^2 + q_{1}(x) S_{0} + r_{2}(x)
+ \frac{d S_{0}}{dx} 
&= 0,
\end{align}
and
\begin{equation}
\label{eq:Riccati-gen-3}
\big(2 S_{-1} + q_{0}(x) \big) S_{m + 1}
+ \sum_{j = 0}^m S_{m - j} S_{j} + q_{1}(x) S_m + \frac{d S_m}{dx} = 0
\quad (m \geq 1).
\end{equation} 
Eq. \eqref{eq:Riccati-gen-1} has two solutions,
and once we fix one of them,
we can determine $S_m$ for $m \geq 0$ uniquely and recursively
by \eqref{eq:Riccati-gen-2} -- \eqref{eq:Riccati-gen-3}.
Thus we obtain two WKB type solutions of the form
\begin{equation}
\psi(x, \hbar) := \exp \left(\int^x S(x, \hbar) dx\right).
\end{equation}
Here the integral of $S(x,\hbar)$ is defined as the term-wise integral.

Let $\Disc(x)$ be the discriminant of 
\eqref{eq:Riccati-gen-1}, i.e.,
\begin{equation}
\label{eq:discriminantQ}
\Disc(x) := \big\{q_{0}(x)\big\}^2 - 4 r_{0}(x) 
= \big\{y_+(x) - y_-(x)\big\}^2.
\end{equation}
Here
$y_{\pm}(x)$ are two solutions of $y^2 + q_0(x) y + r_0(x) = 0$.
A point $a \in {\mathbb C}$ is called a 
{{turning point}} of \eqref{eq:2nd-ODE} 
if it satisfies $\Disc(a) = 0$. If it is a simple zero, 
we say it is a {{simple}} turning point. 
A simple pole $a \in {\mathbb C}$ of
$\Disc(x)$ is called a {simple-pole type turning point}
(cf.\,\cite{Ko2}).
We also call $\infty \in {\mathbb P}^1$ is a turning point 
(resp., simple-pole type turning point) 
if $X=0$ is a zero (resp., simple-pole) of $\Disc(1/X)  / X^4$ 
(i.e. if $X=0$ is a turning point (resp., simple-pole type turning point) 
of the differential equation 
obtained by the coordinate change $X = 1/x$ from \eqref{eq:2nd-ODE}). 

Together with the turning points, 
Stokes curves, which is defined by
\begin{equation}
\Im \int^x_{a} \sqrt{\Disc(x)} dx = 0,
\end{equation}
where $a$ is a turning point 
or a simple-pole type turning point,
play a central role in the exact WKB analysis.
They are used to describe the region where
Borel summability of WKB type solutions holds, and
to give a connection formulas among the Borel sum of
WKB type solutions (see \cite[Section 2]{KT98}).
Although we do not discuss any such analytic properties
of WKB type solutions in this paper, turning points
and Stokes curves are still useful to visualize
a path of integration to define Voros coefficients.
We will show an example of Stokes curves 
in \S\ref{sec:weber}. 

It is more convenient for the exact WKB analysis
if the second order linear differential equation in question is
represented in the so-called SL-form, i.e.,
the second order linear differential equation with no first order term
is more convenient (the word ``SL'' comes from the fact that
the monodromy matrices of the equation of SL-form belong to
${\rm{SL}}(2, \mathbb{C})$).
A gauge transformation
\begin{equation}
\varphi := \exp \left(\frac{1}{2}\int^x q(x, \hbar) dx\right) \psi
\end{equation}
of the unknown function $\psi$
makes \eqref{eq:2nd-ODE} into the SL-form:
\begin{equation}
\label{eq:2nd-ODE:SL}
\left\{\hbar^2 \frac{d^2}{dx^2}
- Q(x, \hbar) \right\}\varphi = 0,
\end{equation}
where
\begin{align}
\label{eq:SL-potential}
 Q(x, \hbar)
& := \frac{1}{4} q(x, \hbar)^2 - r(x, \hbar)
+ \frac{1}{2}\hbar \frac{\partial}{\partial x} q(x, \hbar)
\\
&= \frac{1}{4} q_0(x)^2 - r_{0}(x)
+ \frac{\hbar}{2} \left( q_{0}(x) q_{1}(x)
- 2 r_{1}(x) +  \frac{d q_{0}}{dx}\right)
+ \frac{\hbar^2}{4}
\left({q_{1}(x)}^2 + 2\frac{d q_{1}}{dx} - 4 r_2(x) \right).\notag
\end{align}
Note that the leading term of the potential $Q(x, \hbar)$ is
$\Disc(x) /4$.

\subsection{Voros coefficient}
\label{subsec:Voros-coeff}

A Voros coefficient is defined as a properly
regularized integral of $S(x, \hbar)$ along 
a path connecting singular points of \eqref{eq:2nd-ODE}.
Here, by a singular point of \eqref{eq:2nd-ODE} we mean 
a pole of ${\rm Disc}(x)$ of order greater than or equal to two.
It depends, however, on the situation we consider
how we regularize such an integral.
Fortunately,
all of the examples discussed in \S\ref{sec:weber} 
and \cite{IKT-part2} 
have the property that $S_m(x)$ with $m \geq 1$
is integrable at any singular point of \eqref{eq:2nd-ODE}. 
In this situation we can define Voros coefficients by
\begin{equation}
V_{\gamma_{b_1, b_2}}(\hbar)
:= \int_{\gamma_{b_1, b_2}}
\big( S(x, \hbar) -\hbar^{-1}S_{-1}(x) - S_0(x)\big) dx
= \sum_{m = 1}^{\infty} \hbar^m  \int_{\gamma_{b_1, b_2}} S_m(x) dx,
\end{equation}
where $\gamma_{b_1, b_2}$ is a path from a singular point $b_1$
to a singular point $b_2$.
Note that Voros coefficients only depend on the class
$[\gamma_{b_1, b_2}]$ of paths in the relative  homology group
$$
H_1 \big(\mathbb{P}^1 \setminus
\{\text{Turning points}\},
\{\text{Singular points}\}; \mathbb{Z} \big).
$$
Such an integration contour (or a relative homology class) 
can be understood as a lift of path on $x$-plane 
onto the Riemann surface of $S_{-1}(x)$ 
(i.e., two sheeted covering of $x$-plane)
after drawing branch cuts and distinguishing 
the first and second sheets of the Riemann surface. 
In \S \ref{sec:weber} we will show an example of such contours, 
and compute the associated Voros coefficient.
Our main example in this paper is the Weber equation \eqref{eq:Weber-intro}, 
and other members of a confluent family 
of the Gauss hypergeometric equation will be 
investigated in the second part \cite{IKT-part2}.

\subsection{Eynard-Orantin's topological recursion}
\label{sec:TR}

A starting point of Eynard-Orantin's theory (\cite{EO})
is a spectral curve.
Because we will not discuss the general case in this paper,
we restrict ourselves to the case when a spectral curve is of genus $0$
(see \cite{EO} for the general definition; see also 
Remark \ref{rem:TR} (ii) below).

\begin{dfn} \label{def:spectral-curve}
A spectral curve (of genus $0$) is a pair $(x(z), y(z))$
of non-constant rational functions on $\mathbb{P}^1$, 
such that their exterior differentials 
$dx$ and $dy$ never vanish simultaneously. 
\end{dfn}

Let $R$ be the set of ramification points of $x(z)$,
i.e., $R$ consists of zeros of $dx(z)$ of any order and poles of $x(z)$
whose orders are greater than or equal to two
(here we consider $x$ as a branched covering map
from $\mathbb{P}^1$ to itself).
We further assume that
\begin{itemize}
\item[(A1)]
A function field $\mathbb{C}(x(z), y(z))$ coincides with $\mathbb{C}(z)$.

\item[(A2)]

If $r$ is a ramification point which is a pole of $x(z)$, 
and if $Y(z) = - x(z)^2 y(z)$ is holomorphic near $r$,
then $dY(r) \neq 0$.

\item[(A3)]
All of the ramification points of $x(z)$ are simple,
i.e., the ramification index of each ramification point
is two.

\item[(A4)]
We assume branch points are all distinct,
where a branch point is defined as the image of
a ramification point by $x(z)$.
\end{itemize}

Because of the assumption (A1), we can find an irreducible
polynomial $P(x, y) \in \mathbb{C}[x, y]$ for which $P(x(z), y(z)) = 0$
holds for any $z$.
Hence $(x(z), y(z))$ becomes a regular meromorphic parametrization
in the sense of \cite[\S4]{SWP} of the plane curve
$\mathcal{C}:= \{(x, y) \in \mathbb{C} \mid P(x, y) = 0\}$.
We also call this curve $\mathcal{C}$ a spectral curve
if there is no fear of confusions.

In the assumption (A2), note that the transformation
$(x, y) \mapsto (X, Y) := (1/x, -x^2 y)$ satisfies
$y dx = Y dX$ (hence it is symplectic),
and that $r$ is a ramification point of $X(z)$.
This assumption ensures that
Eynard-Orantin's correlation function $W_{g, n}(z_1, z_2, \cdots, z_n)$
becomes symmetric in their variables.
(See Theorem \ref{thm:TRprop} below.) 
In this assumption we also allow the case when $y(z)$ has a pole
at a ramification point. 

The assumption (A3) is equivalent to saying
that  any zero of $dx(z)$ is simple, and
that the order of any pole of $x(z)$ is less than or equal to two.
From this assumption,
for each $r \in R$, there exists a neighborhood $U$ of $r$
such that $x^{-1}(x(z)) \cap U$ consists of only
two points for $z \in U \setminus \{r\}$.
Let $x^{-1}(x(z)) \cap U = \{z, \overline{z}\}$.
Then, we define a map
$\overline{\; \cdot \rule{0pt}{.7em}\;}: U \setminus\{r\} \rightarrow \mathbb{P}^1$,
called a conjugate map,
by $z \mapsto \overline{z}$. This conjugate map is
characterized by the following three conditions:
\begin{equation}
\text{(a)}~ \overline{z} \neq z\;\text{if}\; z \neq r,
\qquad
\text{(b)}~ x(\overline{z}) = x(z),
\qquad
\text{(c)}~ \lim_{z \rightarrow r} \overline{z} = z.
\end{equation}
The conjugate map can be extended to $U$
as a holomorphic map: if $r\in\mathbb{C}$ is a zero
of $dx(z)$, and $x(z) =
x_0 + x_2 (z - r)^2 + x_3 (z-r)^3 + \cdots$ with a nonzero $x_2$,
then
\begin{equation}
\label{eq:zbar1}
\overline{z} = r - (z-r) - \frac{x_3}{x_2} (z-r)^2 + \cdots
\end{equation}
near $r$.
If $r\in\mathbb{C}$ is a pole of $x(z)$, and
$x(z) = (z-r)^{-2} \big\{ x_0 + x_1 (z - r) + \cdots\big\}$
with a nonzero $x_0$, then
\begin{equation}
\label{eq:zbar2}
\overline{z} = r - (z-r) + \frac{x_1}{x_0} (z-r)^2 + \cdots.
\end{equation}

The assumption (A4) will be imposed for variational formulas of 
\cite[\S 5]{EO}; see also \S \ref{subsec:variational-formula-TR}.

\begin{rem}
\label{rem:conjugate}
The conjugate map is only defined near a ramification point.
In the subsequent sections of this paper, however, we only study the
case when the degree of $P(x, y)$ with respect to $y$ is two (cf. (AQ1) in
\S \ref{subsec:WKB-series-from-TR}).
In this case the degree of $x(z)$ as a rational function is two,
or, in other words, $x: \mathbb{P}^1 \rightarrow \mathbb{P}^1$ is a 2-sheeted
branched covering, and the conjugate map becomes a globally defined rational map
from $\mathbb{P}^1$ to itself.
\end{rem}

\begin{dfn}[{\cite[Definition 4.2]{EO}}]
Eynard-Orantin's correlation function
$W_{g, n}(z_1, \cdots, z_n)$ for $g \geq 0$ and $n \geq 1$
is defined
as a multidifferential\footnote{We borrow this terminology from \cite{DN16}.
We summarize in \S \ref{sec:multidifferential} our notations
on multidifferentials.} on $(\mathbb{P}^1)^n$ 
using the recurrence relation
(called Eynard-Orantin's topological recursion)
\begin{align}
\label{eq:TR}
W_{g, n+1}(z_0, z_1, \cdots, z_n)
&:= \sum_{r \in R}
\Res_{z = r} K_r(z_0, z)
\Bigg[
W_{g-1, n+2} (z, \overline{z}, z_1, \cdots, z_n)
\\
&\qquad\qquad
+
\sum'_{\substack{g_1 + g_2 = g \\ I_1 \sqcup I_2 = \{1, 2, \cdots, n\}}}
W_{g_1, |I_1| + 1} (z, z_{I_1})
W_{g_2, |I_2| + 1} (\overline{z}, z_{I_2})
\Bigg]
\notag
\end{align}
for $2g + n \geq 2$ with initial conditions
\begin{align}
W_{0, 1}(z_0) &:= y(z_0) dx(z_0),
\quad
W_{0, 2}(z_0, z_1) = B(z_0, z_1)
:= \frac{dz_0 dz_1}{(z_0 - z_1)^2}.
\end{align}
Here we set $W_{g,n} \equiv 0$ for a negative $g$,
\begin{equation}
\label{eq:RecursionKernel}
K_r(z_0, z)
:= \frac{1}{2\big(y(z) - y(\overline{z})\big) dx(z)}
\int^{\zeta = z}_{\zeta = \overline{z}} B(z_0, \zeta)
\end{equation}
is a recursion kernel defined near a ramification point $r \in R$,
$\sqcup$ denotes the disjoint union,
and
the prime ${}'$ on the summation symbol in \eqref{eq:TR}
means that we exclude terms for
$(g_1, I_1) = (0, \emptyset)$
and
$(g_2, I_2) = (0, \emptyset)$
(so that $W_{0, 1}$ does not appear) in the sum.
We have also used the multi-index notation:
for $I = \{i_1, \cdots, i_m\} \subset \{1, 2, \cdots, n\}$
with $i_1 < i_2 < \cdots < i_m$, $z_I:= (z_{i_1}, \cdots, z_{i_m})$.
\end{dfn}

Some remarks are necessary here:
\begin{rem}
\label{rem:TR}
\begin{itemize}

\item[(i)]
Eq. \eqref{eq:TR} becomes a recurrence relation with respect to
$2 g + n$. 

\item[(ii)]
In general setting formulated by \cite{EO}, a spectral curve is defined as a 
triplet $(\Sigma, x, y)$ of compact Riemann surface $\Sigma$ 
together with its Torelli marking, and meromorphic functions $x,y$ on $\Sigma$. 
When $\Sigma$ has genus $\ge 1$, $B(z_0, z_1)$ is replaced 
by the fundamental bilinear differential 
(also called the Bergman kernel in the literature)
of the second kind with respect to the Torelli marking.

\item[(iii)]
Eynard and Orantin defined $R$ as a set of zeros of $dx(z)$ in \cite{EO},
and they call an element of $R$ a branch point.
As far as we know, it is \cite{BE} which
pointed out that poles of $x(z)$ of order two or more
also play the same role as zeros of $dx(z)$ in the
topological recursion.
In this paper we use this modified version.
See also \cite{BHLMR, BE-12}, where higher multiple zero of $dx(z)$ is studied.


\end{itemize}
\end{rem}

\begin{thm}
\label{thm:TRprop}
Under the assumptions (A1) -- (A3), we obtain
\begin{itemize}
\item[{\rm{(i)}}]
$W_{g, n}(z_1, \cdots, z_n)$ is a symmetric meromorphic multidifferential.

\item[{\rm{(ii)}}]
Only singular points
of $W_{g, n}(z_1, \cdots, z_n)$ for $2g + n > 2$
(i.e., $(g, n) \neq (0, 1), (0, 2)$)
with respect to each variable
are ramification points.
They are poles with no residue.
\end{itemize}

\noindent
If we further assume that the degree of $x(z)$ is two, then
\begin{itemize}
\item[{\rm{(iii)}}]
The following relation holds for $2g + n \ge 2$:
\begin{equation}
W_{g, n} (z_1, \cdots, z_n) + W_{g, n} (\overline{z_1}, \cdots, z_n) 
= \delta_{g,0} \delta_{n,2} \frac{dx(z_1)dx(z_2)}{(x(z_1)-x(z_2))^2}.
\end{equation}
\end{itemize}
\end{thm}

See \cite[Theorem 4.6]{EO}, \cite[Theorem 4.2]{EO} and
\cite[Theorem4.4]{EO} respectively for the proof. 
(Although the original paper \cite{EO} does not include 
higher order poles of $x(z)$ to the set of ramification points, 
the proof can be given in the same way.)

It turns out that, ramification points given as higher order poles of $x(z)$ 
sometimes do not contribute to the topological recursion; 
namely, the correlation functions $W_{g,n}(z_1,\dots,z_n)$ 
may be holomorphic at such ramification points for each variable $z_i$ 
($i=1,\dots,n$) except for $(g,n)=(0,1)$, and the residue at those points 
in \eqref{eq:TR} may become zero. This happens for examples 
which will be considered in the second paper \cite{IKT-part2}. 
Therefore we introduce the following notion.

\begin{dfn} \label{def:effective-ramification}
A ramification point $r$ is said to be ineffective if 
the correlation functions $W_{g,n}(z_1,\dots,z_n)$ 
for $(g,n) \ne (0,1)$ are holomorphic at $z_i = r$ for each $i=1,\dots,n$. 
A ramification point which is not ineffective is called effective.
The set of effective ramification points is denoted by $R^{\ast}$ $(\subset R)$.
\end{dfn}

The following properties of ineffective ramification points are important in this paper. 
\begin{prop} \label{prop:ineffective}
Under the assumptions (A1) -- (A3), we obtain the following:
\begin{itemize}
\item[\rm (i)] 
Let $r$ be a ramification point. Then, $r$ is an ineffective ramification point
if and only if $(y(z) - y(\overline{z})) dx(z)$ has a pole at $r$.
\item[\rm (ii)]
If $r \in R$ is an ineffective ramification point, 
then the residue at $r$ in \eqref{eq:TR} is zero.
\end{itemize}
\end{prop}

The proof will be given in \S \ref{sec:miscTR} 
(cf. Propositions \ref{prop:ineffective-ramification} and \ref{prop:ineffective-ramification-2}).
The property (i) above implies that an ineffective ramification point often 
appears as a double pole of $x(z)$.

These properties (Theorem \ref{thm:TRprop} and Proposition \ref{prop:ineffective}) 
of the correlation functions are fundamental,
and we often use them without any reference.

\begin{rem}[cf.\,\cite{BE-12}, {\cite[Remark 3.10]{BE}}]
Correlation functions also satisfy
\eqref{eq:TR}, where $K_r(z_0, z)$ is now
replaced by
\begin{equation}
\label{eq:RecursionKernel-D}
K_{D, r}(z_0, z)
:= \frac{1}{\big(y(z) - y(\overline{z})\big) dx(z)}
\int_{\zeta \in D(z)} B(z_0, \zeta)
\end{equation}
with any divisor $D(z) = [z] - \sum_{j = 1}^m \nu_j [\beta_j]$
with $\beta_j \in \mathbb{P}^1 \setminus R^\ast$ 
and $\nu_j \in \mathbb{C}$ satisfying $\sum_{j = 1}^m \nu_j = 1$.
(See \S\ref{sec:multidifferential} for the definition of the integral with a divisor.) 
\end{rem}

\subsection{Free energy through the topological recursion}
\label{subsec:free-energy}

The $g$-th free energy $F_g$ ($g\geq 0$) is a complex number
defined for the spectral curve,
and one of the most important objects in Eynard-Orantin's theory.
It is also called a symplectic invariant since it is 
``almost'' invariant under symplectic transformations
of spectral curves (see \cite{EO-13} for the details). 

\begin{dfn}[{\cite[Definition 4.3]{EO}}]
For $g \geq 2$, the $g$-th free energy $F_g$ is defined by
\begin{equation}
\label{def:Fg2}
F_g := \frac{1}{2- 2g} \sum_{r \in R} \Res_{z = r}
\big[\Phi(z) W_{g, 1}(z) \big]
\quad (g \geq 2),
\end{equation}
where $\Phi(z)$ is a primitive of $y(z) dx(z)$. 
The free energies $F_0$ and $F_1$ for $g=0$ and $1$ 
are also defined, but in a different manner 
(see \cite[\S 4.2.2 and \S 4.2.3]{EO} for the definition). 
\end{dfn}
Note that the right-hand side of \eqref{def:Fg2} does not
depend on the choice of the primitive
because $W_{g, 1}$ has no residue at each ramification point
(see Theorem \ref{thm:TRprop} (ii)).

In applications (and in our article), the generating series
\begin{equation} \label{eq:total-free-energy}
F := \sum_{g = 0}^{\infty} \hbar^{2g-2} F_g
\end{equation}
of $F_g$'s is crucially important. 
We also call the generating series \eqref{eq:total-free-energy} 
the free energy of the spectral curve. 

\subsection{Variational formulas for the correlation functions}
\label{subsec:variational-formula-TR}

In \S \ref{sec:weber} and \cite{IKT-part2} we will consider a family 
of spectral curves parametrized by complex parameters. 
For our purpose, we briefly recall the variational formulas obtained 
by \cite[\S 5]{EO} which describe the differentiation 
of the correlation functions $W_{g,n}$ and the free energies $F_g$ 
with respect to the parameters.

Suppose that we have given a family 
$(x_\varepsilon(z), y_\varepsilon(z))$
of spectral curves parametrized by a complex parameter 
$\varepsilon$ which lies on a certain domain $U \subset {\mathbb C}$ 
such that  
\begin{itemize}
\item 
$x_\varepsilon(z), y_\varepsilon(z)$ depend 
holomorphically on $\varepsilon \in U$. 
\item 
$x_\varepsilon(z), y_\varepsilon(z)$ 
satisfy the assumptions (A1) -- (A4) 
for any $\varepsilon \in U$. 
\item 
The cardinality of the set $R_\varepsilon$ of 
ramification points of $x_\varepsilon(z)$ is constant on $\varepsilon \in U$
(i.e. ramification points of $x_\varepsilon(z)$ 
are distinct for any $\varepsilon \in U$).
\end{itemize}
Then, the correlation functions $W_{g,n}(z_1, \dots, z_n; \varepsilon)$ 
and the $g$-th free energy $F_g(\varepsilon)$ defined from the spectral curve 
$(x_\varepsilon(z), y_\varepsilon(z))$
are holomorphic in $\varepsilon \in U$ 
as long as $z_i \notin R_\varepsilon$ for any $i=1,\dots,n$. 

In order to formulate a variational formula for correlation functions, 
we need to introduce the notion of ``differentiation with fixed $x$". 
For a meromorphic differential $\omega(z; \varepsilon)$ on ${\mathbb P}^1$, 
which depends on $\varepsilon$ holomorphically, define 
\begin{equation}
\delta_{\varepsilon} \, \omega(z; \varepsilon)  
:= \left( 
\frac{\partial}{\partial \varepsilon} \omega(z_{\varepsilon}(x); \varepsilon) 
\right) \biggl|_{x=x_{\varepsilon}(z)} 
\quad (z \notin R_\varepsilon), 
\end{equation}
where $z_{\varepsilon}(x)$ is (any branch of) the inverse function of 
$x = x_{\varepsilon}(z)$ which is defined away from branchpoints 
(i.e. points in $x_{\varepsilon}(R_\varepsilon)$). 
In \cite{EO} the notation 
$\delta_{\Omega} \, \omega(z; \varepsilon) \big|_{x(z)}$ 
is used for $\delta_{\varepsilon} \omega(z; \varepsilon)$ defined above.
Such differentiation $\delta_{\varepsilon}$ can be generalized 
to multidifferentials in an obvious way. 
Then, under these assumptions, the variational formula is formulated as follows.

\begin{thm}[{\cite[Theorem 5.1]{EO}}] \label{thm:VariationFormula}
In addition to the above conditions, for any $\varepsilon \in U$, 
we further assume that  
\begin{itemize}
\item 
If $r_\varepsilon \in R_\varepsilon$ is a zero of 
$dx_\varepsilon(z)$, then the functions
$\partial x_\varepsilon/ \partial \varepsilon$ and 
$\partial y_\varepsilon/ \partial \varepsilon$ are holomorphic
(as functions of $z$) at $r_\varepsilon$, 
and $dy_\varepsilon(z)$ does not vanish
(as a differential of $z$) at $r_\varepsilon$.
\item 
If $r_\varepsilon \in R_\varepsilon$ is a pole of $x_\varepsilon(z)$ 
with an order greater than or equal to two, then 
\[
\frac{\Omega_\varepsilon(z) \, B(z_1, z) \, B(z_2 , z)}
{dy_\varepsilon(z) dx_\varepsilon(z)}
\]
is holomorphic (as a differential in $z$) at $r(\varepsilon)$, where 
\begin{equation} \label{eq:Omega}
\Omega_\varepsilon(z) := 
\frac{\partial y_\varepsilon}{\partial \varepsilon}(z) \, dx(z)
- \frac{\partial  x_\varepsilon}{\partial \varepsilon}(z) \, dy(z).
\end{equation}
\item 
There exist a path $\gamma$ in $\mathbb{P}^1$ passing through
no ramification point and a function $\Lambda_\varepsilon (z)$ 
holomorphic in a neighborhood of $\gamma$ for which the following holds.
\begin{equation}
\Omega_\varepsilon(z) = 
\int_{\zeta \in \gamma} \Lambda_\varepsilon(\zeta) \, B(z, \zeta).
\end{equation} 
\end{itemize}
Then, $W_{g,n}(z_1, \dots, z_n; \varepsilon)$ 
and $F_g(\varepsilon)$ defined from the spectral curve 
$(x_\varepsilon(z), y_\varepsilon(z))$
satisfy the following relations:  
\begin{itemize}
\item[{\rm{(i)}}]
For $2g + n \geq 2$, 
\begin{equation}
\delta_{\varepsilon} \, W_{g, n} 
(z_1, \cdots, z_n; \varepsilon)
= \int_{\zeta \in \gamma} \Lambda_\varepsilon(\zeta) \,
W_{g, n + 1}(z_1, \cdots, z_n, \zeta; \varepsilon) 
\end{equation}
holds on $\varepsilon \in U$ as long as each of $z_1, \cdots, z_n$ satisfies 
$z_i \notin R_\varepsilon$.

\item[{\rm{(ii)}}]
For $g \geq 1$,
\begin{equation}
\frac{\partial F_g}{\partial \varepsilon}(\varepsilon)
= \int_{\gamma}\Lambda_\varepsilon(z) \, W_{g, 1}(z;\varepsilon)
\end{equation}
holds on $\varepsilon \in U$.


\end{itemize}

\end{thm}

See \cite[\S 5.1]{EO} 
(based on the Rauch's variation formula; see \cite{KK} for example) 
for the proof.
We note that, since we modify the definition of the topological recursion 
by adding higher order poles of $x(z)$ as ramification point, 
we also need to require the second condition in the above claim. 

In examples discussed in \S \ref{sec:weber} and \cite{IKT-part2}, 
we will use the variational formula in a situation that
$\Lambda_\varepsilon(z) = \Lambda$ is a constant function 
(which is also independent of the parameter $\varepsilon$). 
In that case, applying the above formulas iteratively, we have 
\begin{equation} \label{eq:iterative-variation}
\frac{\partial^n F_g}{\partial \varepsilon^n}(\varepsilon)
 = \Lambda^n 
\int_{\zeta_1 \in \gamma}  \cdots \int_{\zeta_n \in \gamma}
W_{g, n}(\zeta_1, \cdots, \zeta_n; \varepsilon) .
\end{equation}


\section{Quantization of spectral curves}
\label{sec:quantization-spectral-curve}

\subsection{The WKB-type formal series from the topological recursion}
\label{subsec:WKB-series-from-TR}

With correlation functions $W_{g, n}(z_1, \cdots, z_{n})$
of a spectral curve $(x(z),y(z))$ we associate
a formal series defined by
\begin{align}
\label{eq:wave-z}
\varphi(z;\underline{\nu}, \hbar)
&:= \exp \Bigg[ \frac{1}{\hbar}\int_{\zeta_1 \in D(z; \underline{\nu})} W_{0, 1}(\zeta_1)
	+ \frac{1}{2!} \int_{\zeta_1\in D(z ; \underline{\nu})} 
	\int_{\zeta_2 \in D(z ; \underline{\nu})}
	\left( W_{0, 2}(\zeta_1, \zeta_2) - 
	\frac{dx(\zeta_1) \, dx(\zeta_2)}{(x(\zeta_1) - x(\zeta_2))^2} \right) \\
& \qquad\qquad
\left.
 + \sum_{m = 1}^{\infty} \hbar^m
 \left\{ \sum_{\substack{2g + n - 2 = m \\ g \geq 0, \, n \geq 1}}
 \frac{1}{n!} \int_{\zeta_1 \in D(z ; \underline{\nu})} 
 \cdots \int_{\zeta_n \in D(z ; \underline{\nu})} 
 W_{g, n}(\zeta_1, \cdots, \zeta_n)
 \right\} \right],
\notag
\end{align}
where
\begin{equation}
D(z;\underline{\nu}) = [z] - \sum_{\beta \in B} \nu_\beta [\beta]
\end{equation}
is a divisor on $\mathbb{P}^1$ with a finite set $B \subset \mathbb{P}^1 \setminus R^\ast$
and $\underline{\nu} = (\nu_{\beta})_{\beta \in B}$ which is a tuple of 
complex numbers satisfying $\sum_{\beta \in B}\nu_{\beta} = 1$
(see \S\ref{sec:multidifferential}, for the integral with a divisor). 
Precisely speaking, the above integrals of $W_{0,1}$ and $W_{0,2}$ are not 
well-defined when $\beta \in B$ is a singular point of these differentials. 
In the situation, we define these integrals through a certain regularization technique 
(e.g., method used in Remark \ref{rem:regularization}). 
We do not consider the regularization in detail because 
only their $z$-derivatives are important in the following discussion
(in particular, in the proof of Theorem \ref{thm:WKB-Wg,n}).

Note that the quantization by using the divisor with the parameter 
$\underline{\nu}$ was first  introduced by \cite{BE}. 
This $\varphi(z;\underline{\nu}, \hbar)$ has the same form as
WKB type solutions \eqref{eq:WKB-type}, and
it is known that $\psi(x; \underline{\nu}, \hbar) = \varphi(z(x); \underline{\nu}, \hbar)$
for some class of spectral curves
with some specified divisors $D(z, \underline{\nu})$
satisfies a linear ordinary differential equation,
where $z(x)$ is a inverse function of $x(z)$.
Such a linear ordinary differential equation
is called a quantum curve of a spectral curve in question.

A typical example of a quantization is the Airy curve
\begin{equation} 
x(z) = z^2, \quad 
y(z) = z \qquad (z \in \mathbb{P}^1)
\end{equation}
(so that $P(x(z), y(z)) = 0$ with $P(x, y)= y^2 - x$).
In this case we choose $D = [z] - [\infty]$, and
the corresponding quantum curve is the Airy equation
\begin{equation}
\left\{ \hbar^2 \frac{d^2}{dx^2} - x \right\} \psi = 0
\end{equation}
(\cite{GS, Zhou12}, see also \cite{EO}, \cite{EO-08}).
Another example is
\begin{equation} 
x(z) = z^2, \quad 
y(z) = 1/z \qquad (z \in \mathbb{P}^1)
\end{equation}
called Bessel curve in \cite{DN16}. 
In this case $P(x(z), y(z)) = 0$ with $P(x, y)= x y^2 - 1$,
and the corresponding quantum curve is obtained in \cite{DN16}:
\begin{equation}
\left\{ \hbar^2 \frac{d^2}{dx^2} - \frac{1}{x} \right\} \psi = 0.
\end{equation}

A systematic study of a quantization is done in \cite{BE},
and explicit forms of quantum curves are given 
for spectral curves satisfying the admissibility assumption 
(\cite[Definition 2.7]{BE}).
The spectral curve \eqref{eq:Gauss-curve-intro}
of the Gauss hypergeometric equation in the SL-form 
(which we will discuss in \cite{IKT-part2}), 
however, does not satisfy the admissibility condition.
In this section we study quantizations without the admissibility condition, 
while we restrict ourselves to the case when the degree of
a polynomial $P(x, y)$ with respect to $y$ is two.

To be more precise,
we assume the following conditions
in addition to (A1) -- (A4):

\begin{description}
\item[{\rm(AQ1)}]
The rational functions $(x(z), y(z))$ satisfy
$P(x(z), y(z)) = 0$ with an irreducible polynomial
$P(x, y) = p_0(x) y^2 + p_1(x) y + p_2(x) \in \mathbb{C}[x, y]$,
where $p_0(x)$ is a nonzero polynomial.

\item[{\rm{(AQ2)}}]
The differential $(y(z) - y(\bar{z})) dx(z)$ does not vanish on ${\mathbb P}^1 \setminus R$.
\end{description}

By the assumption (AQ1), the conjugate map is now a rational map
defined globally (cf. Remark \ref{rem:conjugate}).
Hence the assumption (AQ2) makes sense.
We also note that $y(z)$ and $y(\overline{z})$ are 
two solutions of $P(x(z), y) = 0$. 
These solutions are distinct when $z \notin R$.



\begin{rem}
The assumption (AQ2) ensures that
two solutions $y(z)$ and $y(\overline{z})$ of $P(x(z), y) = 0$
merge to each other only at ramification points.
An example of a spectral curve which does not satify
(AQ2) is
\begin{equation}
P(x, y) = y^2 - x^3 - x^2,
\quad
x(z) = - 1 + z^2,
\quad
y(z) = z (z^2 - 1).
\end{equation}
Because $R = \{0, \infty\}$ and $\overline{z} = -z$,
we have $y(1) = y(\overline{1}) = 0$,
whereas $z = 1$ is not a ramification point of $x(z)$.
Thus (AQ2) is violated.
At  $(x(z), y(z))|_{z = 1} = (0, 0)$, a curve defined by $\{(x, y) \mid P(x, y) = 0\}$ has
a normal crossing. Hence the point $x(1) = 0$ should become a
double turning point of its quantum curve (if it exists).
As illustrated by this example, (AQ2) excludes multiple turning points.
A quantum curve with a double turning point is discussed in \cite{IS}, 
where the infinitely many $\hbar$-correction terms appear 
in the quantum curve. It was also shown in \cite{IS} that 
the correction terms are related to asymptotic expansion of Painlev\'e transcendents.
\end{rem}

\subsection{An index of the defining polynomials}

To state assumptions for the divisor $D(z; \underline{\nu})$ in \eqref{eq:wave-z},
we introduce an index at $x_0 \in \mathbb{P}^1$ by
\begin{equation}
\rho(x_0; P) :=
\begin{cases}
 \ord_{x_0} Q_0(x) & (x_0 \neq \infty),\\
 \ord_{0} Q_{0}^{(\infty)}(x) & (x_0 = \infty)
\end{cases}
\end{equation}
for a polynomial $P(x, y) = p_0(x) y^2 + p_1(x) y + p_2(x) \in {\mathbb C}[x, y]$, where
\begin{equation} \label{eq:Q0-in-section-3}
Q_0(x) := \frac{1}{4}\left(\frac{p_1(x)}{p_0(x)}\right)^2 -  \frac{p_2(x)}{p_0(x)}
\quad\text{and}\quad
Q_0^{(\infty)}(x) := \frac{1}{x^4} Q_0({1}/{x}).
\end{equation}
Here $\ord_{x_0} g(x)$ for a
rational function $g(x)$ is defined as an integer $p$ for which
\begin{equation}
g(x) = (x - x_0)^{p} \big(c_0 + c_1 (x - x_0) + \cdots \big), \quad c_0
 \neq 0
\end{equation}
holds as the Laurent expansion of $g(x)$ at $x_0$.

\begin{rem}
\begin{itemize}
\item[(i)]
This index $\rho(x_0; P)$ is related to the type of singularities
of the quantum curve.
We have expected that the leading term with respect to $\hbar$
of the quantum curve is given by the second order differential
operator $\mathscr{L} := (p_0(x))^{-1} P(x, \hbar (d/dx))$,
where we use
the normal ordering with respect to $x$ and $\hbar (d/dx)$.
Then  $Q_0(x)$ is the leading term with respect to $\hbar$ of the
potential of the SL-form of $\mathscr{L}$
(cf. \eqref{eq:2nd-ODE:SL}).
It is well-known that
\begin{itemize}
\item[(a)]
$x_0 \in \mathbb{P}^1$  is a regular singular point of
the SL-form of $\mathscr{L}$
iff $\rho(x_0; P) = -1, -2$. \\
($x_0$ with $\rho(x_0; P) = -1$ is a simple-pole type turning point
in the WKB analysis, we we have mentioned in \S \ref{sec:WKB-and-Voros}.)

\item[(b)]
$x_0 \in \mathbb{P}^1$  is an irregular singular point of
the SL-form of $\mathscr{L}$
iff $\rho(x_0; P) \leq -3$.
\end{itemize}

\item[(ii)]
Our index $\rho(x_0; P)$ remains invariant under a symplectic
transformation of the form
$(x, y) \mapsto (X, Y) = (x, y + g(x))$
for some function $g(x)$, or, a gauge transformation
$\mathscr{L} \mapsto e^{- \int g(x) dx} \circ \mathscr{L} \circ   e^{\int g(x) dx}$.
\end{itemize}
\end{rem}

\begin{prop}
\label{prop:order-of-ydx}
We assume (A1) -- (A4), (AQ1) and (AQ2).
Then, for $\alpha \in \mathbb{P}^1$,
\begin{equation}
\ord_{\alpha} \left[(y(z) - y(\overline{z})) \frac{dx}{dz}(z)\right] =
\begin{cases}
\dfrac{1}{2}\rho(x(\alpha); P) & (\alpha \not\in R),\\[+.5em]
\rho(x(\alpha); P) + 1 & (\alpha \in R).
\end{cases}
\end{equation}
\end{prop}

\begin{proof}
Because $y(z)$ and $y(\overline{z})$ are two solutions of $P(x(z), y) = 0$, we have
\begin{equation}
\label{eq:y-vs-ybar}
\big\{y(z) - y(\overline{z})\big\}^2
=
\left.
\left\{
\left(\frac{p_1(x)}{p_0(x)}\right)^2 - 
4 \, \frac{p_2(x)}{p_0(x)}\right\}\right|_{x = x(z)}
= 4 Q_0(x(z)) .
\end{equation}
Hence if $\alpha$ is a regular point of $x(z)$, we have
\begin{equation}
\label{eq:order-result1}
\ord_{\alpha} \big[ y(z) - y(\overline{z})\big]
= \frac{1}{2} \ord_{x(\alpha)} Q_0(x)
= \frac{1}{2} \rho(x(\alpha); P).
\end{equation}
If $\alpha$ is a pole of $x(z)$, we utilize a transformation
\begin{equation}
\label{eq:symplectic-transform}
X(z) = \frac{1}{x(z)}, \quad Y(z) = -x(z)^2 y(z).
\end{equation}
By this transformation, we have
\begin{equation}
\label{eq:symplectic-transform2}
( y(z) - y(\overline{z}) ) dx(z)
=
( Y(z) - Y(\overline{z}) ) dX(z).
\end{equation}
Furthermore
$Y(z)$ and $Y(\overline{z})$ are two solutions of
$P^{(\infty)} (X, Y) = 0$, where
\begin{equation}
P^{(\infty)} (X, Y) = P(1/X, -X^2 Y)
= X^4 p_0(1/X) Y^2 - X^2 p_1(1/X) Y + p_2(1/X).
\end{equation}
Hence
\begin{equation}
\big\{Y(z) - Y(\overline{z})\big\}^2
=
\left.
\frac{1}{X^4}
\left\{
\left(\frac{p_1(1/X)}{p_0(1/X)}\right)^2 - 4
\frac{p_2(1/X)}{p_0(1/X)}\right\}\right|_{X = X(z)}
= 4 Q_0^{(\infty)}(X(z)) .
\end{equation}
Thus
\begin{equation}
\label{eq:order-result2}
\ord_{\alpha} \big[ Y(z) - Y(\overline{z})\big]
= \frac{1}{2} \ord_{0} Q_0^{(\infty)}(x)
= \frac{1}{2} \rho(x(\alpha); P).
\end{equation}

\noindent
Now we divide the proof into four cases.
\begin{itemize}
\item[(a)]
$\alpha \not\in R$ and $\alpha$ is a regular point of $x(z)$:
In this case $\ord_{\alpha} [dx/dz] = 0$ follows, and hence
\begin{equation}
\ord_{\alpha} \left[(y(z) - y(\overline{z})) \frac{dx}{dz}(z)\right]
=
\ord_{\alpha} \left[ y(z) - y(\overline{z})\right]
\overset{\eqref{eq:order-result1}}{=} \frac{1}{2} \rho(x(\alpha); P).
\end{equation}

\item[(b)]
$\alpha \not\in R$ and $\alpha$ is a simple pole of $x(z)$:
Since $\alpha$ is a regular point of $X(z)$ with $X(\alpha) = 0$ and $X'(\alpha) \neq 0$,
we obtain
\begin{eqnarray}
\label{eq:order-relation:tmp1}
\ord_{\alpha} \left[(y(z) - y(\overline{z})) \frac{dx}{dz}(z)\right]
&\overset{\eqref{eq:symplectic-transform2}}{=}&
\ord_{\alpha} \left[(Y(z) - Y(\overline{z})) \frac{dX}{dz}(z)\right]\\
&=&
\ord_{\alpha} \big[ Y(z) - Y(\overline{z})\big]\notag\\
&\overset{\eqref{eq:order-result2}}{=}&
\frac{1}{2} \rho(x(\alpha); P).
\notag
\end{eqnarray}

\item[(c)]
$\alpha \in R$ and $\alpha$ is a simple zero of $dx(z)$:
In this case
\begin{align}
\ord_{\alpha} \big[x(z) - x(\alpha)\big] = 2
\quad\text{and}\quad
\ord_{\alpha} \left[\frac{dx}{dz}(z)\right] = 1
\end{align}
hold.
Hence
\begin{align}
\ord_{\alpha} \big[\big(y(z) - y(\overline{z})\big) \frac{dx}{dz}(z)\big]
&= \frac{1}{2}
\rho(x(\alpha); P) \times \big( \ord_{\alpha} [x(z) - x(\alpha)]\big)
+ 1\notag\\
&= \rho(x(\alpha); P) + 1.
\notag
\end{align}

\item[(d)]
$\alpha \in R$ and $\alpha$ is a double pole of $x(z)$:
Since
\begin{align}
\ord_{\alpha} \big[X(z) - X(\alpha)\big] = 2
\quad\text{and}\quad
\ord_{\alpha} \left[\frac{dX}{dz}(z)\right] = 1,
\end{align}
we have
\begin{align}
\ord_{\alpha} \left[\big(y(z) - y(\overline{z})\big) \frac{dx}{dz}(z)\right]
&=
\ord_{\alpha} \left[\big(Y(z) - Y(\overline{z})\big) \frac{dX}{dz}(z)\right]\\
&= \frac{1}{2}
\rho(x(\alpha); P) \times \big( \ord_{\alpha} [X(z) - X(\alpha)]\big)+ 1\notag\\
&= \rho(x(\alpha); P) + 1.
\notag
\end{align}
\end{itemize}
\end{proof}

\begin{rem}
\label{rem:ineffective}
It follows from Proposition \ref{prop:order-of-ydx} that,
if $r \in R$ satisfies $\rho (x(r); P) \leq -2$, then
$r$ is an ineffective ramification point
(cf. Proposition \ref{prop:ineffective}).
\end{rem}


\subsection{Quantum curves}
\label{subsec:main}

In order to give our theorem in a clear form, we introduce
\begin{equation}
\Sing (P) :=\big\{b \in \mathbb{P}^1 \mid \rho (b; P) \leq -2\}, 
\qquad
\Sing_2 (P) :=\big\{b \in \mathbb{P}^1 \mid \rho (b; P) = -2\}
\end{equation}
for a polynomial $P(x, y) = p_0(x) y^2 + p_1(x) y + p_2(x) \in {\mathbb C}[x, y]$.
We also use the following symbols:
\begin{align}
\Delta(z) &:= y(z) - y(\overline{z}), \\
C_{\beta} &:= \Res_{z = \beta} \Delta(z) dx(z).
\end{align}

\begin{thm}
\label{thm:WKB-Wg,n}
We assume (A1) -- (A4), (AQ1) and (AQ2).
Let
\begin{equation} \label{eq:divisor-theorem}
D(z; \underline{\nu}) := [z] - \sum_{\beta \in B}\nu_{\beta} [\beta]
\end{equation}
be a divisor on $\mathbb{P}^1$, where
$B := x^{-1}(\Sing (P))$,
and $\underline{\nu} =(\nu_{\beta})_{\beta \in B}$ is a tuple of complex numbers satisfying
\begin{equation}
\sum_{\beta \in B} \nu_{\beta} = 1.
\end{equation}
Then, 
\begin{equation}
\psi(x;\underline{\nu}, \hbar) := \varphi (z(x); \underline{\nu}, \hbar), 
\end{equation}
where
$z(x)$ denotes an inverse function of $x(z)$ and
$\varphi(z, \hbar)$ is defined by \eqref{eq:wave-z}
with the integration divisor \eqref{eq:divisor-theorem},
is a WKB type formal solution of $\hat{P} \psi = 0$,
where $\hat{P}$ is defined by
\begin{equation}
\label{eq:quantization}
\hat{P} :=
\hbar^2\frac{d^2}{dx^2} + q(x, \hbar) \hbar \frac{d}{dx} + r(x, \hbar)
\end{equation}
with
\begin{equation} \label{eq:leading-coeff-of-quantum-curve}
q(x, \hbar) = q_0(x) + \hbar q_1(x),
\quad
r(x, \hbar) = r_0(x) + \hbar r_1(x) + \hbar^2 r_2(x)
\end{equation}
whose leading terms are respectively given by
\begin{equation}
q_0(x) = \frac{p_1(x)}{p_0(x)},
\quad
r_0(x) = \frac{p_2(x)}{p_0(x)}
\end{equation}
and their lower order terms are determined by
\begin{align}
\label{eq:WKB-Wg,n:q1}
x'(z) q_1(x(z))
&=
- \frac{\Delta'(z)}{\Delta(z)}
+ \frac{2}{z - \overline{z}}
- \sum_{\beta \in B}
\frac{\nu_{\beta} + \nu_{\overline{\beta}}}{z - \beta},\\
\label{eq:WKB-Wg,n:r1}
x'(z) r_1(x(z))
&= \frac{1}{2} x'(z) \frac{\partial q_0}{\partial x}\Big|_{x = x(z)}
+ \frac{1}{2} x'(z) q_0(x(z)) q_1(x(z))
+ \frac{1}{2} \Delta(z)
\sum_{\beta \in B}
\frac{\nu_{\beta} - \nu_{\overline{\beta}}}{z - \beta},
\\
\label{eq:WKB-Wg,n:r2}
x'(z) r_2(x(z))
&= \Delta(z) \sum_{\beta \in B_1}
 \frac{\nu_{\beta}\nu_{\overline{\beta}}}{C_{\beta}} \frac{1}{z - \beta}.
\end{align}
Here we set $B_1 := x^{-1}(\Sing_2(P))$.
\end{thm}

We call the equation $\hat{P}\psi = 0$ given by Theorem \ref{thm:WKB-Wg,n}
a quantum curve of the spectral curve.

\begin{rem}
\label{rem:MainThm}
\begin{itemize}
\item[(i)]
The set $B$ may contain the infinity.
In Theorem \ref{thm:WKB-Wg,n},
we use the convention that if $\beta \in B$ is the infinity,
$1/(z - \beta)$ in $q_1$, $r_1$ and $r_2$ is zero.
In what follows, we use this convention.

\item[(ii)]
It follows from Proposition \ref{prop:order-of-ydx} that
\begin{align}
B &= \text{the set of poles of $\Delta(z) dx(z)$}, \\
B_1 &=\text{the set of simple poles of $\Delta(z) dx(z)$}.
\end{align}
Hence, $C_\beta \ne 0$ for $\beta \in B_1$.

\item[(iii)]
The set $B$ may contain ramification points; 
this happens in examples treated in \cite{IKT-part2}.
It follows, however, from the definition of $\Sing (P)$ and 
Remark \ref{rem:ineffective} 
that ramification points contained in $B$ are ineffective.
Hence, $W_{g,n}$'s are holomorphic there, 
and hence, the formal series \eqref{eq:wave-z} is well-defined.
Note also that, since ${\rm Sing}(P)$ and ${\rm Sing}_2(P)$
are defined in terms of order of $Q_0(x)$ or $Q_{\infty}(x)$,
they are closed by the conjugate map;
that is, if $\beta \in B$ (resp. $\beta \in B_1$), then
$\overline{\beta} \in B$ (resp. $\overline{\beta} \in B_1$).

 \item[(iv)]
It is not obvious that $q_1(x)$, $r_1(x)$ and $r_2(x)$ defined by 
\eqref{eq:WKB-Wg,n:q1} -- \eqref{eq:WKB-Wg,n:r2} are rational functions in $x$. 
But for the examples which will be considered in \S4 and \cite{IKT-part2}, 
$q_1(x)$, $r_1(x)$ and $r_2(x)$ become rational functions in $x$. 

\item[(v)]
Because the degree of $x(z)$ is two by our assumption (AQ1),
we have two inverse functions of $x(z)$.
By Theorem \ref{thm:WKB-Wg,n},
we can construct two WKB type formal solutions
from these two inverse functions,
and they form a basis of the space of formal solutions of $\hat{P}\psi = 0$.

\end{itemize}
\end{rem}

\subsection{Preparations for the proof }
Our proof of Theorem \ref{thm:WKB-Wg,n} is based on
``diagonal specialization" which allows us to relate 
the topological recursion \eqref{eq:TR}  
to the recursive relations 
\eqref{eq:Riccati-gen-1}--\eqref{eq:Riccati-gen-3}
satisfied by the coefficients of WKB solution 
(\cite{DM}; see also \cite{BE, IS}).
In this subsection we prepare several statements for the proof.

In the remaining part of this section
we denote $D(z; \underline{\nu})$ by $D(z)$ for simplicity.
We also assume (A1) -- (A4), (AQ1) and (AQ2) 
as we have assumed in Theorem \ref{thm:WKB-Wg,n}.

Let us define
\begin{align}
\hat{T}(z, \hbar)
=
\frac{1}{\hbar} \hat{T}_{-1}(z) + \hat{T}_0(z) + \hbar \hat{T}_1(z) + \cdots
:= \frac{d}{dz} \log \varphi(z; \underline{\nu}, \hbar),
\end{align}
that is,
\begin{align}
\label{eq:def:That}
&\quad
\hat{T}(z, \hbar) dz =
\frac{1}{\hbar}W_{0,1}(z)
+ \int_{\zeta \in D(z)}
\left( W_{0, 2}(z, \zeta) - \frac{dx(z) \, dx(\zeta)}{(x(z) - x(\zeta))^2} \right)
\\
&\quad\qquad
+ \sum_{m = 1}^{\infty} \hbar^m
\left\{ \sum_{\substack{2g + n - 2 = m \\ g \geq 0, \, n \geq 1}}
\frac{1}{(n-1)!} \int_{\zeta_1 \in D(z)} \cdots \int_{\zeta_{n-1} \in D(z)}
W_{g, n}(z, \zeta_1, \cdots, \zeta_{n-1})
\right\}.
\notag
\end{align}
Then, in order to prove Theorem \ref{thm:WKB-Wg,n}, 
it is enough to show that $\hat{T}(z,\hbar)$ satisfies 
the Riccati equation associated with $\hat{P}$ 
after a change of variable $x = x(z)$.
Comparing the coefficients, we have
\begin{equation}
\label{eq:That_m}
\hat{T}_m(z) dz
=
\begin{cases}
\displaystyle
W_{0,1}(z)  & (m = -1),\\[2ex]
\displaystyle
\int_{\zeta \in D(z)}
\left\{
W_{0, 2}(z, \zeta)
- \frac{dx(z)dx(\zeta)}{(x(z) - z(\zeta))^2}
\right\}
 & (m = 0),\\[2ex]
\displaystyle
\sum_{\substack{2g + n - 2 = m \\ g \geq 0, n\geq 1}}
\frac{1}{(n-1)!} G_{g, n}(z, \cdots, z)
& (m \geq 1),
\end{cases}
\end{equation}
where
\begin{equation}
\label{def:G_gn}
G_{g, n}(z_0, z_1, \cdots, z_{n-1})
:=
\begin{cases}
W_{g, 1}(z_0) & (n = 1),\\[2ex]
\displaystyle
\int_{\zeta_1 \in D(z_1)}  \cdots  \int_{\zeta_{n-1} \in D(z_{n-1})}
W_{g, n}(z_0, \zeta_1, \cdots, \zeta_{n-1}) & (n \geq 2)
\end{cases}
\end{equation}
for $2g + n \geq 2$.
Note that $G_{g, n}(z_0, z_1, \cdots, z_{n-1})$
is a meromorphic differential (1-form) in the first variable $z_0$,
and is a meromorphic functions (0-form) in the remaining variables $z_1, \cdots, z_{n-1}$.
We will derive a recurrence relation satisfied by $\hat{T}_m$ by using several properties
(especially the topological recursion \eqref{eq:TR}) of $W_{g, n}$.

For the later convenience, we set
$G_{g, n} = W_{g, n} = 0$ for either $g < 0$ or $n \leq 0$.

\begin{prop}
\label{prop:G-properties-01}
\quad
\begin{itemize}
\item[{\rm{(i)}}]
$G_{g, n}(z_0, z_1, \cdots, z_{n-1})$ for $2g + n \geq 3$ is holomorphic
in each variable on $\mathbb{P}^1 \setminus (R^{*} \cup B)$.
All singular points of them are poles.

\item[{\rm{(ii)}}]
For $(g, n) = (0, 2)$,
\begin{align}
\label{eq:G_02}
G_{0, 2}(z_0, z_1)
=
\left(
\frac{1}{z_0 - z_1}
- \sum_{\beta \in B} \frac{\nu_{\beta}}{z_0 - \beta}
\right) dz_0.
\end{align}
\end{itemize}
\end{prop}

\begin{proof}
Because only singularities of $W_{g, n}$ for $2g + n \geq 3$ are poles whose
residues vanish, we obtain (i).
By a straightforward computation using
$W_{0, 2}(z_0, z_1) = B(z_0, z_1) = dz_0 dz_1/(z_0 - z_1)^2$,
we can prove (ii).
\end{proof}

\begin{prop}
\label{prop:G-properties-02}
\quad
\begin{itemize}
\item[{\rm{(i)}}]
$\displaystyle
G_{0, 2}(\overline{z_0}, z_1)
=
\left(
\frac{1}{z_0 - \overline{z_1}}
- \sum_{\beta \in B} \frac{\nu_{\beta}}{z_0 - \overline{\beta}}
\right) dz_0.$

\item[{\rm{(ii)}}]
For $2g + n \geq 3$,
\begin{equation}
G_{g, n}(\overline{z_0}, z_1, \cdots, z_{n-1})
= - G_{g, n}(z_0, z_1, \cdots, z_{n-1}).
\end{equation}
\end{itemize}
\end{prop}

\begin{proof}
By Theorem \ref{thm:TRprop} (iii), we have
\begin{equation}
W_{0, 2}(z_0, z_1) + W_{0, 2}(\overline{z_0}, z_1) 
= \frac{dx(z_0)dx(z_1)}{\, (x(z_0) - x(z_1))^2}.
\end{equation}
Because $W_{0, 2}$ is symmetric in its variables, we also have
\begin{equation}
W_{0, 2}(z_0, z_1) + W_{0, 2}(z_0, \overline{z_1}) 
= \frac{dx(z_0)dx(z_1)}{\, (x(z_0) - x(z_1))^2}.
\end{equation}
From these two relations, we have
\begin{equation}
W_{0, 2}(z_0, \overline{z_1}) = W_{0, 2}(\overline{z_0}, z_1).
\end{equation}
Thus
\begin{align}
G_{0, 2}(\overline{z_0}, z_1)
&= \int_{\zeta_1 \in  D(z_1)} W_{0, 2}(\overline{z_0}, \zeta_1)
= \int_{\zeta_1\in D(z_1)} W_{0, 2}(z_0, \overline{\zeta_1})
\\
&= \sum_{\beta \in B} \nu_{\beta}
\left(\int^{z_1}_{\beta} \frac{d\overline{\zeta_1}}{(z_0 - \overline{\zeta_1})^2} \right) dz_0
= \sum_{\beta \in B} \nu_{\beta}
\left(\int^{\overline{z_1}}_{\overline{\beta}}
\frac{d\xi_1}{(z_0- \xi_1 )^2} \right) dz_0\notag\\
&=
\sum_{\beta \in B} \nu_{\beta}
\left(\frac{1}{z_0 - \overline{z_1}} - \frac{1}{z_0 - \overline{\beta}}\right) dz_0
=
\left(
\frac{1}{z_0 - \overline{z_1}}
- \sum_{\beta \in B} \frac{\nu_{\beta}}{z_0 - \overline{\beta}}
\right) dz_0.\notag
\end{align}
The relation (ii) is a direct consequence of Theorem \ref{thm:TRprop} (iii).
\end{proof}

We now state two propositions which give recurrence relations of $\{G_{g, n}\}$.
These relations play key roles in the proof of Theorem \ref{thm:WKB-Wg,n}.

\begin{prop}
\label{prop:recursion-01}
We have
\begin{align}
\label{eq:prop:recursion-00}
G_{1, 1}(z_0) &= \frac{B(z_0, \overline{z_0})}{\Delta(z_0) dx(z_0)}
\end{align}
and
\begin{align}
\label{eq:prop:recursion-01}
& G_{0, 3}(z_0, z_1, z_2)\\
&\qquad
=
G_{0, 2}(z_0, z_1)
\left\{
\frac{G_{0, 2}(\overline{z_0}, z_2)}{\Delta(z_0) dx(z_0)}
-
\frac{G_{0, 2}(\overline{z_1}, z_2)}{\Delta(z_1) dx(z_1)}
\right\}
+ \frac{G_{0, 2}(z_0, \overline{z_1}) G_{0, 2}(\overline{z_1}, z_2)}{\Delta (z_1) dx(z_1)}
\notag\\
&\qquad\qquad
+
G_{0, 2}(z_0, z_2)
\left\{
\frac{G_{0, 2}(\overline{z_0}, z_1)}{\Delta(z_0) dx(z_0)}
-
\frac{G_{0, 2}(\overline{z_2}, z_1)}{\Delta(z_2) dx(z_2)}
\right\}
+ \frac{G_{0, 2}(z_0, \overline{z_2}) G_{0, 2}(\overline{z_2}, z_1)}{\Delta (z_2) dx(z_2)}
\notag\\
&\qquad\qquad
- \sum_{\beta \in B_1} \frac{\nu_{\beta}
 \nu_{\overline{\beta}}}{C_{\beta}} \big\{G_{0, 2}(z_0, \beta)
 -  G_{0,2}(z_0, \overline{\beta}) \big\}.
\notag
\end{align}
\end{prop}

\begin{prop}
\label{prop:recursion-02}
For $2g + n \geq 3$
we have
\begin{align}
\label{eq:prop:recursion-02}
G_{g, n + 1}(z_0, z_I)
&
=
\frac{1}{\Delta(z_0) dx (z_0)}
\int_{\zeta_1 \in D(z_1)} \cdots \int_{\zeta_n \in D(z_n)}
W_{g-1, n + 2}(z_0, \overline{z_0}, \zeta_I)
\\
&\qquad
+
\sum_{j = 1}^n
G_{0, 2}(z_0, z_j)
\left\{
\frac{G_{g, n}(\overline{z_0}, z_{I \setminus \{j\}})}{\Delta(z_0) dx(z_0)}
-
\frac{G_{g, n}(\overline{z_j}, z_{I \setminus \{j\}})}{\Delta(z_j) dx(z_j)}
\right\}
\notag\\
&\qquad
+
\sum_{j = 1}^n
\left\{
\frac{G_{0, 2}(\overline{z_0}, z_j) G_{g, n} (z_0, z_{I \setminus \{j\}})}{\Delta(z_0) dx(z_0)}
+
\frac{G_{0, 2}(z_0, \overline{z_j}) G_{g, n} (\overline{z_j}, z_{I \setminus \{j\}})}{\Delta(z_j) dx(z_j)}
\right\}
\notag\\
&\qquad
+
\frac{1}{\Delta(z_0) dx (z_0)}
\sum_{\substack{g_1 + g_2 = g, \\ I_1 \sqcup I_2 = I, \\
 2g_1 + |I_1| \geq 2, \\ 2g_2 + |I_2| \geq 2}}
G_{g_1, |I_1| + 1}(z_0, z_{I_1})
G_{g_2, |I_2| + 1}(\overline{z_0}, z_{I_2}),
\notag
\end{align}
where $I = \{1, 2, \cdots, n\}$.
\end{prop}

\begin{proof}[Proof of Proposition \ref{prop:recursion-01}]
It follows from the topological recursion \eqref{eq:TR} for $(g, n) = (1, 0)$ that
\begin{equation}
G_{1, 1}(z_0)
=
\sum_{r \in R^{*}} \Res_{\xi = r}
\left[\frac{G_{0, 2}(z_0, \xi)}{\Delta(\xi) dx(\xi)} W_{0, 2}(\xi, \overline{\xi})\right]
\end{equation}
holds (cf. \eqref{eq:RecursionKernel-D} and Proposition \ref{prop:ineffective} (ii)).
Here we remind the readers that the set $R^\ast$ consists of
the effective ramification points in the sense of
Definition \ref{def:effective-ramification}.
Because, under the assumption (AQ2),
the singular points of the integrand are contained in
$R^{*} \cup \{z_0\}$, the residue theorem gives
\begin{equation}
G_{1, 1}(z_0)
=
- \Res_{\xi = z_0}
\left[\frac{G_{0, 2}(z_0, \xi)}{\Delta(\xi) dx(\xi)} W_{0, 2}(\xi, \overline{\xi})\right]
= \frac{W_{0, 2}(z_0, \overline{z_0})}{\Delta (z_0) dx(z_0)}.
\end{equation}
Here we have used \eqref{eq:G_02} to compute the residue.
Similarly, it follows from the topological recursion \eqref{eq:TR} for $(g, n) = (0, 2)$ that
\begin{equation}
G_{0, 3}(z_0, z_1, z_2)
= \sum_{r \in R^{*}} \Res_{\xi = r}
\big[ g_{0, 3}(z; z_0, z_1, z_2) \big]
\end{equation}
where
\begin{align}
g_{0, 3}(\xi; z_0, z_1, z_2)
&=
K_{D(\xi)}(z_0, \xi)
\big\{ G_{0, 2}(\xi, z_1) G_{0, 2}(\overline{\xi}, z_2)
+  G_{0, 2}(\xi, z_2) G_{0, 2}(\overline{\xi}, z_1) \big\}
\\
&=
\frac{G_{0, 2}(z_0, \xi)}{\Delta(\xi) dx(\xi)}
\big\{ G_{0, 2}(\xi, z_1) G_{0, 2}(\overline{\xi}, z_2)
+  G_{0, 2}(\xi, z_2) G_{0, 2}(\overline{\xi}, z_1) \big\}.
\nonumber
\end{align}
Since
$g_{0, 3}(\xi; z_0, z_1, z_2)$ is holomorphic in $\xi$ except
for points in $R^{*} \cup  B \cup \{z_0, z_1, \overline{z_1},  z_2, \overline{z_2}\}$,
the residue theorem gives
\begin{align}
G_{0, 3}(z_0, z_1, z_2)
& = - \sum_{r \in \{z_0, z_1, \overline{z_1}, z_2, \overline{z_2}\}}
\Res_{\xi = r} \big[ g_{0, 3}(\xi; z_0, z_1, z_2) \big]
- \sum_{\beta \in B}
\Res_{\xi = \beta} \big[ g_{0, 3}(\xi; z_0, z_1, z_2) \big].
\end{align}
By Proposition \ref{prop:G-properties-02},
we find that
$g_{0, 3}(\xi; z_0, z_1, z_2)$ has simple poles
at $z_0$, $z_1$, $\overline{z_1}$,
$z_2$, $\overline{z_2}$, and its residues are given respectively by
\begin{align*}
&\Res_{\xi = z_0} \big[ g_{0, 3}(\xi; z_0, z_1, z_2) \big]\\
&\qquad\quad
=
- \frac{1}{\Delta(z_0) dx(z_0)} \big\{
G_{0, 2}(z_0, z_1) G_{0, 2}(\overline{z_0}, z_2)
+ G_{0, 2}(\overline{z_0}, z_1) G_{0, 2}(z_0, z_2)
\big\},
\notag\\
&\Res_{\xi = z_1} \big[ g_{0, 3}(\xi; z_0, z_1, z_2) \big]
+\Res_{\xi = \overline{z_1}} \big[ g_{0, 3}(\xi; z_0, z_1, z_2) \big]\\
&
\qquad\quad
=
\frac{G_{0, 2}(z_0, z_1) G_{0, 2}(\overline{z_1}, z_2)}{\Delta(z_1) dx(z_1)}
-
\frac{G_{0, 2}(z_0, \overline{z_1}) G_{0, 2}(\overline{z_1}, z_2)}{\Delta(z_1) dx(z_1)},
\notag\\
&\Res_{\xi = z_2} \big[ g_{0, 3}(\xi; z_0, z_1, z_2) \big]
+\Res_{\xi = \overline{z_2}} \big[ g_{0, 3}(\xi; z_0, z_1, z_2) \big]\\
&
\qquad\quad
=
\frac{G_{0, 2}(z_0, z_2) G_{0, 2}(\overline{z_2}, z_1)}{\Delta(z_2) dx(z_2)}
-
\frac{G_{0, 2}(z_0, \overline{z_2}) G_{0, 2}(\overline{z_2}, z_1)}{\Delta(z_2) dx(z_2)}.
\notag
\end{align*}
Next we compute the residue at $\beta \in B$.
By Proposition \ref{prop:G-properties-01} (ii) and
Proposition \ref{prop:G-properties-02} (i),
we have the Laurent expansion as
\begin{equation}
G_{0, 2}(\xi, z_1) G_{0, 2}(\overline{\xi}, z_2)
= \frac{\nu_{\beta} \nu_{\overline{\beta}}}{(\xi - \beta)^2}
\bigl(1 +O(\xi - \beta) \bigr)  (d\xi)^2
\nonumber
\end{equation}
near $\xi = \beta$
(note that $\overline{\beta} \in B$ since $B$ is closed by the conjugate map;
see Remark \ref{rem:MainThm} (iii)).
Hence, if $\Delta(\xi) dx(\xi)$ has a double or higher order pole at
$\xi = \beta$, then
\begin{equation}
\label{eq:prop:recursion-01:tmp1}
\frac{G_{0, 2}(z, \xi) G_{0, 2}(\xi, z_1)
G_{0, 2}(\overline{\xi}, z_2)}{\Delta(\xi) dx(\xi)}
\end{equation}
is holomorphic at $\xi = \beta$ and its residue is zero.
If $\Delta(\xi) dx(\xi)$ has a simple pole at $\beta$,
i.e., if $\beta \in B_1$ (cf. Remark \ref{rem:MainThm} (ii)), then
\begin{align}
\eqref{eq:prop:recursion-01:tmp1}
&=
\frac{\nu_{\beta} \nu_{\overline{\beta}}}{C_{\beta}}
G_{0, 2}(z, \beta)
\left(
\frac{1}{\xi - \beta} + O(1)
\right) d\xi
\notag
\end{align}
when $\xi \to \beta$.
Hence we obtain
\begin{align}
\Res_{\xi = \beta}\left[
\frac{G_{0, 2}(z, \xi) G_{0, 2}(\xi, z_1) G_{0, 2}
(\overline{\xi}, z_2)}{\Delta(\xi) dx(\xi)}\right]
=
\begin{cases}
 0 &  (\beta\not\in B_1),\\[2ex]
\dfrac{\nu_{\beta} \nu_{\overline{\beta}}}{C_{\beta}}
G_{0, 2}(z, \beta) & (\beta \in B_1).
\end{cases}
\end{align}
Then, the desired equality \eqref{eq:prop:recursion-01} follows from
\begin{equation} \label{eq:residue-at-beta-relation}
\sum_{\beta \in B_1} \dfrac{\nu_{\beta} \nu_{\overline{\beta}}}{C_{\beta}}
G_{0, 2}(z, \beta) =
\sum_{\beta \in B_1} 
\dfrac{\nu_{\beta} \nu_{\overline{\beta}}}{C_{\overline{\beta}}}
G_{0, 2}(z, \overline{\beta})
\end{equation}
(cf. Remark \ref{rem:MainThm} (iii)) and
\begin{equation} \label{eq:C-beta-bar}
C_{\overline{\beta}}
= \Res_{z = \overline{\beta}} \Delta(z) dx(z)
= \Res_{z = \beta} \Delta(\overline{z}) dx(\overline{z})
= - \Res_{z = \beta}  \Delta(z) dx(z)
= - C_{\beta} \quad (\beta \in B).
\end{equation}
This completes the proof.
\end{proof}

\begin{proof}[Proof of Proposition \ref{prop:recursion-02}]
The topological recursion gives
\begin{align}
 G_{g, n + 1}(z_0, z_{I})
&=
\sum_{r \in R^*}
\Res_{\xi = r}
\left[
f_1(\xi; z_0, z_I) + f_2(\xi; z_0, z_I) + f_3(\xi; z_0, z_I)
\right].
\end{align}
Here
\begin{align}
f_1(\xi; z_0, z_I)
&=
\frac{G_{0, 2}(z_0, \xi)}{\Delta (\xi) dx(\xi)}
\int_{\zeta_1 \in D(z_1)} \cdots \int_{\zeta_n \in D(z_n)}
W_{g-1, n + 2} (\xi, \overline{\xi}, \zeta_I),
\\
f_2(\xi; z_0, z_I)
&=
\frac{G_{0, 2}(z_0, \xi)}{\Delta(\xi) dx (\xi)}
\sum_{j = 1}^n
\left\{
G_{0, 2} (\xi, z_j) G_{g, n}(\overline{\xi}, z_{I \setminus \{j\}})
+
G_{0, 2} (\overline{\xi}, z_j) G_{g, n}(\xi, z_{I \setminus \{j\}})
\right\},
\\
f_3(\xi; z_0, z_I)
&=
\frac{G_{0, 2}(z_0, \xi)}{\Delta(\xi) dx (\xi)}
\sum_{\substack{g_1 + g_2 = g, \\ I_1 \sqcup I_2 = I}}''
G_{g_1, |I_1| + 1}(\xi, z_{I_1})
G_{g_2, |I_2| + 1}(\overline{\xi}, z_{I_2}),
\end{align}
where $\sum''$ in $f_3(\xi; z_0, z_I)$
means that we take the sum for
\begin{equation}
 2g_1 + |I_1| \geq 2 \quad\text{and}\quad 2g_2 + |I_2| \geq 2.
\end{equation}

Because $f_1(\xi; z_0, z_I)$ and $f_3(\xi; z_0, z_I)$ are holomorphic
except for $R^{*}  \cup \{z_0\}$, the residue theorem gives
\begin{align}
&\sum_{r \in R^*}
\Res_{\xi = r}
\left[
f_1(\xi; z_0, z_I) + f_3(\xi; z_0, z_I)
\right]
\\
&\qquad
=
\frac{1}{\Delta (z_0) dx(z_0)}
\int_{\zeta_1 \in D(z_1)} \cdots \int_{\zeta_n \in D(z_n)}
W_{g-1, n + 2} (z_0, \overline{z_0}, \zeta_I)
\notag\\
&\qquad\qquad
+
\frac{1}{\Delta(z_0) dx (z_0)}
\sum_{\substack{g_1 + g_2 = g, \\ I_1 \sqcup I_2 = I}}''
G_{g_1, |I_1| + 1}(z_0, z_{I_1})
G_{g_2, |I_2| + 1}(\overline{z_0}, z_{I_2}).
\notag
\end{align}

Although $G_{0, 2}(\xi, z_j)$ has a simple pole at $\xi = \beta \in B$,
$f_2(\xi; z_0, z_I)$ is holomorphic at $\xi = \beta$ because
$\Delta(\xi) dx(\xi)$ has a pole at $\beta$
(cf. Remark \ref{rem:MainThm} (ii)).
It is also holomorphic near $\xi = \overline{\beta}$ with $\beta \in B$.
Hence the singular points of $f_2(\xi; z_0, z_I)$ are contained in
$ R^{*} \cup \{z_0\} \cup \{z_j, \overline{z_j}\}_{j = 1}^n$.
Then, by the residue theorem we obtain
\begin{align}
&\sum_{r \in R^*}
\Res_{\xi = r}
f_2(\xi; z_0, z_I)
\\
&\qquad
=
\frac{1}{\Delta(z_0) dx (z_0)}
\sum_{j = 1}^n
\left\{
G_{0, 2} (z_0, z_j) G_{g, n}(\overline{z_0}, z_{I \setminus \{j\}})
+
G_{0, 2} (\overline{z_0}, z_j) G_{g, n}(z_0, z_{I \setminus \{j\}})
\right\}
\notag\\
&\qquad\qquad
-
\sum_{j = 1}^n
\frac{G_{0, 2}(z_0, z_j)}{\Delta(z_j) dx (z_j)}
G_{g, n}(\overline{z_j}, z_{I \setminus \{j\}})
+
\sum_{j = 1}^n
\frac{G_{0, 2}(z_0, \overline{z_j})}{\Delta(z_j) dx (z_j)}
G_{g, n}(\overline{z_j}, z_{I \setminus \{j\}}).
\notag
\end{align}
Here the first term comes from the residue at $z_0$,
the second term from the residue at $z_j$,
and the third term from the residue at $\overline{z_j}$.
In computing the residue at $\overline{z_j}$, we have used the relation
$\Delta(\overline{z_j}) dx(\overline{z_j})
= - \Delta(z_j) dx(z_j)$.
\end{proof}

Let us consider the diagonal specialization of $G_{g, n}$.
For the purpose, we define
\begin{align}
\label{def:H_gn}
H_{g, n}(z)
:=\frac{1}{(n-1)!} \, G_{g, n}(z_0, z_1, \cdots, z_{n-1})
\Big|_{z_0 = z_1 = \cdots = z_{n-1} = z}.
\end{align}
This is a meromorphic differential in $z$.

\begin{prop}
\label{prop:recursion-H-01}
\begin{align}
\label{eq:prop:recursion-H-00}
H_{1, 1}(z)
&= \frac{W_{0, 2}(z, \overline{z})}{\Delta(z) dx(z)},
\\
\label{eq:prop:recursion-H-01}
H_{0, 3}(z)
&=
\left.
\frac{\partial}{\partial z_0}
\left(
\frac{G_{0, 2}(\overline{z_0}, z)}{\Delta (z_0) x'(z_0)}
\right)\right|_{z_0 = z}
+ \frac{G_{0, 2}(z, \overline{z}) G_{0, 2}(\overline{z}, z)}{\Delta (z) dx(z)}
\\
&\qquad
- \sum_{\beta \in B_1} \frac{\nu_{\beta}
 \nu_{\overline{\beta}}}{2 C_{\beta}}
\left(
G_{0, 2}(z, \beta) - G_{0, 2}(z, \overline{\beta})
\right),
\notag\\
\label{eq:prop:recursion-H-02}
H_{g, n + 1}(z)
&=
- \frac{1}{\Delta(z){x'(z)}
} \left[
\frac{\partial}{\partial z} H_{g-1, n + 2}(z)
-
\left.
\frac{\partial}{\partial z_0}
\left(\frac{G_{g-1, n + 2}(z_0, z, \cdots, z)}{(n + 1)!}\right)
\right|_{z_0 = z}
\right]
\\
&\qquad
- \frac{d}{d z}
\left(
\frac{1}{\Delta (z) x'(z)}
\right)
H_{g, n}(z)
- \frac{1}{ \Delta (z) x'(z)}
\frac{\partial}{\partial z_0}
\left.
\left(
\frac{G_{g, n}(z_0, z, \cdots, z)}{(n - 1)!}\right)
\right|_{z_0 = z}
\notag\\
&\qquad
+ \frac{G_{0, 2}(\overline{z}, z) - G_{0, 2}(z, \overline{z})}{\Delta (z) dx (z)}
H_{g, n}(z)
\notag
\\
&\qquad
- \frac{1}{\Delta(z) dx(z)}
\sum_{\substack{g_1 + g_2 = g, \\ n_1 + n_2 = n, \\ 2 g_1 + n_1 \geq 2,
 \\ 2g_2 + n_2 \geq 2}} H_{g_1, n_1 + 1}(z) H_{g_2, n_2 + 1}(z)
\qquad (2g + n \geq 3).
\notag
\end{align}
Here and in what follows, we use the following convention:
For a meromorphic differential $f(z) dz$, its $z$-derivative means
\[
\frac{\partial}{\partial z} \left( f(z) dz \right) := \frac{\partial f}{\partial z}(z) dz.
\]
\end{prop}

\begin{proof}
Because $G_{0, 2}(z_j, z_k)$ is singular along $z_j = z_k$,
we need a careful treatment for the terms which contain it.
For example, the first term in the right-hand side of
\eqref{eq:prop:recursion-01} becomes
\begin{align}
&\lim_{z_0, z_1, z_2 \rightarrow z}
\left[
G_{0, 2}(z_0, z_1)
\left(
\frac{G_{0, 2}(\overline{z_0}, z_2)}{\Delta(z_0) dx(z_0)}
-
\frac{G_{0, 2}(\overline{z_1}, z_2)}{\Delta(z_1) dx(z_1)}
\right) \right]
\\
&\quad
=
\lim_{z_0, z_1, z_2 \rightarrow z}
\left[
\left(
\frac{dz_0}{z_0 - z_1}
+ (\text{holomorphic along $z_0 = z_1$})
\right)
\times
\left(
\frac{G_{0, 2}(\overline{z_0}, z_2)}{\Delta(z_0) dx(z_0)}
-
\frac{G_{0, 2}(\overline{z_1}, z_2)}{\Delta(z_1) dx(z_1)}
\right) \right]
\notag\\
&\quad
= \left.
\frac{\partial}{\partial z_0}
\left(
\frac{G_{0, 2}(\overline{z_0}, z)}{\Delta(z_0) x'(z_0)}
\right)\right|_{z_0 = z}
\notag
\end{align}
after the diagonal specialization.
In the same manner, we can compute the third term in the right-hand side
of \eqref{eq:prop:recursion-01}, and obtain the same result.
This proves \eqref{eq:prop:recursion-H-01}.
We use the same computation to obtain  \eqref{eq:prop:recursion-H-02}
together with Theorem \ref{thm:TRprop} (iii) and the relation
\begin{align}
&
\frac{1}{dx(z)}
\int_{\zeta_1 \in D(z)} \cdots \int_{\zeta_{n} \in D(z)}
\frac{W_{g-1, n+2}(z, \overline{z}, \zeta_{I})}{n!}
\\
&\qquad\quad
=
- \frac{1}{dx(z)} \int_{\zeta_1 \in D(z)} \cdots \int_{\zeta_{n} \in D(z)}
\frac{W_{g-1, n+2}(z, z, \zeta_{I})}{n!}
\notag\\
&\qquad\quad
=
\frac{1}{x'(z)} \left\{ - \frac{\partial}{\partial z} H_{g-1, n+2}(z)
+ \left.\frac{\partial}{\partial z_0}
\left( \frac{G_{g-1, n+2} (z_0, z, \cdots, z)}{(n+1)!}\right)\right|_{z_0 = z}
\right\}. \notag
\end{align}
\end{proof}

\begin{prop}
\label{prop:That}
\begin{align}
\label{eq:prop:That:0}
 \hat{T}_0(z) dz
&= -G_{0, 2}(\overline{z}, z), \\
\label{eq:prop:That:1}
\hat{T}_1(z)
&=
- \frac{1}{\Delta(z) x'(z)}\frac{d \hat{T}_0}{dz}(z)
- \frac{\partial}{\partial z} \left(\frac{1}{\Delta(z) x'(z)}\right) \hat{T}_0(z)
\\
&\qquad
- \frac{G_{0, 2}(z, \overline{z}) + G_{0, 2}(\overline{z}, z)}{\Delta (z)dx (z)} \hat{T}_0(z)
- \frac{1}{\Delta(z) x'(z)} {\hat{T}_0(z)}^2
\notag\\
&\qquad
- \sum_{\beta  \in B_1}
\frac{\nu_{\beta} \nu_{\overline{\beta}}}{2 C_{\beta}}
\frac{G_{0, 2}(z, \beta) - G_{0, 2}(z, \overline{\beta})}{dz},
\notag
\\
\label{eq:prop:That:m+1}
 \hat{T}_{m + 1}(z)
&= - \frac{1}{\Delta(z) x'(z)} \frac{d \hat{T}_{m}}{dz}(z)
- \frac{\partial}{\partial z}\left(\frac{1}{\Delta(z)x'(z)}\right) \hat{T}_m(z)
\\
&\qquad
- \frac{G_{0, 2}(\overline{z}, z) + G_{0, 2}(z, \overline{z})}{\Delta(z) dx(z)} \hat{T}_m(z)
- \frac{1}{\Delta(z)x'(z)}
\sum_{j =0}^m \hat{T}_{m-j}(z) \hat{T}_j(z)
\quad (m \geq 1).
\notag
\end{align}
\end{prop}

\begin{proof}
It follows from the definition \eqref{eq:That_m} of
$\hat{T}_m(z)$ and Theorem \ref{thm:TRprop} (iii) that
\begin{equation}
\hat{T}_0(z) dz =
\int_{\zeta \in D(z)}
\left(
W_{0, 2}(z, \zeta)
- \frac{dx(z)dx(\zeta)}{(x(z) - x(\zeta))^2}
\right)
= - \int_{\zeta \in D(z)} W_{0, 2}(\overline{z}, \zeta)
= - G_{0, 2}(\overline{z}, z).
\notag
\end{equation}
Thus we obtain \eqref{eq:prop:That:0}.

We may write $\hat{T}_m$ for $m \ge 1$ as
\begin{equation}
\label{eq:prop:That:tmp1}
\hat{T}_m(z) dz
= \sum_{\substack{2g + n - 2 = m, \\ g \geq 0, n \geq 1}}H_{g, n}(z).
\end{equation}
Then, by using Proposition \ref{prop:recursion-H-01},
we find
\begin{align}
\hat{T}_1(z)  dz
& =  H_{0, 3}(z) + H_{1, 1}(z)
\\
& =
\left.
\frac{\partial}{\partial z_0}
\left(
\frac{G_{0, 2}(\overline{z_0}, z)}{\Delta (z_0) x'(z_0)}
\right)\right|_{z_0 = z}
+ \frac{G_{0, 2}(z, \overline{z}) G_{0, 2}(\overline{z}, z)}{\Delta (z) dx(z)}
 \notag\\
 & \qquad
 - \sum_{\beta \in B_1} \frac{\nu_{\beta}
 \nu_{\overline{\beta}}}{2C_{\beta}} 
 \left( G_{0, 2}(z, \beta) - G_{0, 2}(z, \overline{\beta}) \right)
+ \frac{W_{0, 2}(z, \overline{z})}{\Delta(z) dx(z)}.
 \notag
\end{align}
Because
\[
W_{0, 2}(z, \overline{z})
= W_{0, 2}(\overline{z}, z)
= \left. \frac{\partial}{\partial z_1}
\left(G_{0, 2}(\overline{z}, z_1)\right) \right|_{z_1 = z} dz
\]
holds, we conclude that
\begin{align}
\hat{T}_1(z)  dz
&=
\frac{\partial}{\partial z}
\left(\frac{1}{\Delta (z) x'(z)}\right)
G_{0, 2}(\overline{z}, z)
+
\frac{1}{\Delta (z) x'(z)}
\frac{\partial}{\partial z}
\left(
G_{0, 2}(\overline{z}, z)
\right)
\\
&\qquad\quad
+ \frac{G_{0, 2}(z, \overline{z}) G_{0, 2}(\overline{z}, z)}{\Delta (z) dx(z)}
- \sum_{\beta \in B_1} \frac{\nu_{\beta}
 \nu_{\overline{\beta}}}{2C_{\beta}} 
 \left( G_{0, 2}(z, \beta) - G_{0, 2}(z, \overline{\beta}) \right)
\notag
\\
&=
\frac{\partial}{\partial z}
\left(\frac{1}{\Delta (z) x'(z)}\right)
G_{0, 2}(\overline{z}, z)
+
\frac{1}{\Delta (z) x'(z)}
\frac{\partial}{\partial z}
\left(
G_{0, 2}(\overline{z}, z)
\right)
\notag\\
&\qquad\quad
+ \frac{G_{0, 2}(z, \overline{z}) 
+ G_{0, 2}(\overline{z}, z)}{\Delta (z) dx(z)}G_{0, 2}(\overline{z}, z)
- \frac{G_{0, 2}(\overline{z}, z)^2}{\Delta (z) dx(z)}
\notag\\
&\qquad\quad
- \sum_{\beta \in B_1} \frac{\nu_{\beta}
 \nu_{\overline{\beta}}}{2C_{\beta}} 
 \left( G_{0, 2}(z, \beta) - G_{0, 2}(z, \overline{\beta}) \right).
\notag
\end{align}
Then, the equality \eqref{eq:prop:That:1} follows from this equality
together with \eqref{eq:prop:That:0}.

Similarly, for $m \geq 1$, \eqref{eq:That_m} gives
\begin{align}
\hat{T}_{m + 1}(z)  dz
&= \sum_{\substack{2g + n -2 = m + 1, \\ g \geq 0, n \geq 1} }
H_{g, n}(z)
= \sum_{\substack{2g + n -2 = m , \\ g \geq 0, n \geq 0} }
H_{g, n + 1}(z)
\\
&=
- \frac{1}{\Delta(z) x'(z) }
\frac{\partial }{\partial z_0}
\left(
\sum_{\substack{2g + n -2 = m , \\ g \geq 0, n \geq 0} }
H_{g-1, n+2}(z_0)
\right.
\notag\\
&
\left. \left.
-
\sum_{\substack{2g + n -2 = m , \\ g \geq 0, n \geq 0} }
\frac{G_{g-1, n + 2}(z_0, z, \cdots, z)}{(n+1)!}
+
\sum_{\substack{2g + n -2 = m , \\ g \geq 0, n \geq 0} }
\frac{G_{g, n}(z_0, z, \cdots, z)}{(n-1)!}
\right)
\right|_{z_0 = z}
\notag\\
&\qquad
-\frac{d}{dz} \left(\frac{1}{\Delta(z)x'(z)}\right) \hat{T}_m(z) dz
+ \frac{G_{0, 2}(\overline{z}, z) - G_{0, 2}(z, \overline{z})}{\Delta(z) dx(z)}
\hat{T}_m(z) dz
\notag\\
&\qquad
- \frac{1}{\Delta(z) dx(z)}
\sum_{\substack{2g + n -2 = m , \\ g \geq 0, n \geq 0} }
\sum_{\substack{g_1 + g_2 = g, \\ n_1 + n_2 = n, \\ 2 g_1 + n_1 \geq 2,
 \\ 2g_2 + n_2 \geq 2}} H_{g_1, n_1 + 1}(z) H_{g_2, n_2 + 1}(z).
\notag
\end{align}
From this expression we obtain \eqref{eq:prop:That:m+1},
because we can compute as
\begin{align}
&\sum_{\substack{2g + n -2 = m , \\ g \geq 0, n \geq 0} }
H_{g-1, n+2}(z_0) +
\sum_{\substack{2g + n -2 = m , \\ g \geq 0, n \geq 0} }
\left(
\frac{G_{g, n}(z_0, z, \cdots, z)}{(n-1)!}
-
\frac{G_{g-1, n + 2}(z_0, z, \cdots, z)}{(n+1)!}
\right)
\\
&\qquad
=
\begin{cases}
\displaystyle
\sum_{\substack{2g + n -2 = m , \\ g \geq 0, n \geq 2} } H_{g,n}(z_0)
& \text{if $m$ is even}  \\[+2.em]
\displaystyle
\sum_{\substack{2g + n -2 = m , \\ g \geq 0, n \geq 2} } H_{g,n}(z_0)
 + G_{\frac{m+1}{2},1}(z_0)
& \text{if $m$ is odd}
\end{cases}
\notag
\\
&\qquad
= \hat{T}_{m}(z_0) dz_0,
\notag
\end{align}
and
\begin{align}
\frac{1}{dx(z)}
\sum_{\substack{2g + n -2 = m , \\ g \geq 0, n \geq 0} }
\sum_{\substack{g_1 + g_2 = g, \\ n_1 + n_2 = n, \\ 2 g_1 + n_1 \geq 2,
 \\ 2g_2 + n_2 \geq 2}} H_{g_1, n_1 + 1}(z) H_{g_2, n_2 + 1}(z)
&=
\frac{1}{x'(z)}
\sum_{\substack{m_1 + m_2 = m, \\ m_1, m_2 \geq 1 }}
\hat{T}_{m_1}(z) \hat{T}_{m_2}(z)
dz
 \\
&=
\frac{1}{x'(z)}
\left(
\sum_{j = 0}^m  \hat{T}_{m-j}(z) \hat{T}_j(z)
- 2 \hat{T}_0(z) \hat{T}_m(z)
\right) dz
\notag \\
&
=
\frac{1}{x'(z)} \sum_{j = 0}^m  \hat{T}_{m-j}(z) \hat{T}_j(z) dz
+ \frac{2 G_{0,2}(\overline{z},z)}{x'(z)} \hat{T}_m(z), \notag
\end{align}
where we set $m_j = 2g_j + (n_j + 1)-2$ for $j = 1, 2$
to rewrite summation.
\end{proof}

\subsection{Proof of Theorem \ref{thm:WKB-Wg,n}}

Now we give a proof of Theorem \ref{thm:WKB-Wg,n}. 
We will compare the recursive relations 
\eqref{eq:prop:That:0}--\eqref{eq:prop:That:m+1} satisfied by 
the functions $\{\hat{T}_m\}_{m \ge -1}$ to those 
\eqref{eq:Riccati-gen-2}--\eqref{eq:Riccati-gen-3} 
satisfied by the expansion coefficients of the WKB solutions 
(which will be denoted by $\{T_m\}_{m \ge -1}$ below)
after the coordinate change $x=x(z)$.

\medskip

{\bf\boldmath{Transformation of the Riccati equation 
\eqref{eq:Riccati-gen} to $z$-variable.}}\quad
Let $\psi(x,\hbar)$ be the WKB solution of $\hat{P} \psi = 0$,
where $\hat{P}$ is given by the formula \eqref{eq:quantization}
with the coefficients specified by 
\eqref{eq:leading-coeff-of-quantum-curve}--\eqref{eq:WKB-Wg,n:r2}.
Let $S(x, \hbar)$ is a logarithmic derivative of $\psi(x, \hbar)$,
and $T(z, \hbar)$ is that of $\psi(x(z), \hbar)$; namely,
\[
S(x(z), \hbar) = \left.\frac{d}{dx} \log \psi(x, \hbar)\right|_{x = x(z)}
= \frac{1}{x'(z)} \frac{d}{dz}\log \psi(x(z), \hbar)
= \frac{1}{x'(z)} T(z, \hbar).
\]
Since $S(x, \hbar)$ satisfies the Riccati equation \eqref{eq:Riccati-gen}
associated with \eqref{eq:quantization},
$T(z, \hbar)$ satisfies
\begin{equation}
\label{eq:Riccati-gen-z}
\hbar^2 \left\{ \frac{d}{dz} T + T^2 \right\}
+ \hbar \left\{ U(z, \hbar) - \hbar \, \frac{x''(z)}{x'(z)}\right\} T
+  V(z, \hbar) = 0,
\end{equation}
where ${}'$ designates derivative with respect to $z$, and we set
\begin{align}
U(z, \hbar) &= U_0(z) + \hbar U_1(z) :=
x'(z) \,  q(x(z), \hbar), \\
V(z, \hbar) &= V_0(z) + \hbar V_1(z) + \hbar^2 V_2(z) :=
(x'(z))^2 \, r(x(z), \hbar).
\end{align}
Hence
\begin{equation}
T(z, \hbar) = \frac{1}{\hbar} T_{-1}(z)  + T_0(z) + \hbar T_1(z) + \cdots
\end{equation}
satisfies \eqref{eq:Riccati-gen-z} if and only if
the following recurrence relations hold:
\begin{align}
\label{eq:Riccati-gen-z--1}
&{T_{-1}}^2 + U_0(z) T_{-1} + V_0(z) = 0,\\
\label{eq:Riccati-gen-z-m+1}
& \left\{ 2 T_{-1}(z) + U_0 (z) \right\} T_{m + 1}
+ \left\{U_1(z) - \frac{x''(z)}{x'(z)}\right\}   T_{m}
\\
&\qquad\qquad\qquad\qquad
+ \sum_{j = 0}^m T_{m-j} T_{j}
+ \frac{d}{dz} T_{m} + V_{m+2}(z) = 0 \quad (m \geq -1),
\notag
\end{align}
where we set $V_ m(z) = 0$ for $m \geq 3$.
These equations determine the coefficients $\{ T_m \}_{m \ge -1}$
of $T(z)$ uniquely.

Our task for the proof of Theorem \ref{thm:WKB-Wg,n} is to show that
\begin{equation} \label{eq:main-equation-quantization}
T_m(z) = \hat{T}_m(z) \quad (m \ge -1).
\end{equation}
Equivalently, we must prove that $\{\hat{T}_m\}_{m \geq 0}$ satisfies
the recurrence relations \eqref{eq:Riccati-gen-z--1} and
\eqref{eq:Riccati-gen-z-m+1}.
The rest of this subsection is devoted to a proof 
of \eqref{eq:main-equation-quantization}.

\medskip

{\bf\boldmath{Proof of $T_{-1}(z) = \hat{T}_{-1}(z)$.}}\quad
Since
\begin{equation}
U_0(z) = {x'(z)} q_0(x(z)), \quad V_0(z) = {x'(z)}^2 r_0(x(z)),
\end{equation}
and $(x(z), y(z))$ satisfies $y(z)^2 + q_0(x(z)) y(z) + r_0(x(z)) = 0$,
we find
\begin{equation}
\hat{T}_{-1}(z) = \frac{W_{0, 1}(z)}{dz} = y(z) x'(z)
\end{equation}
satisfies \eqref{eq:Riccati-gen-z--1}.
Thus we have verified \eqref{eq:main-equation-quantization} for $m = -1$.

\medskip

{\bf\boldmath{Proof of $T_0(z) = \hat{T}_{0}(z)$.}}\quad
Since $y(z)$ and $y(\overline{z})$ are two solutions of $P(x(z), y) = 0$,
\begin{equation}
y(z) + y(\overline{z}) = - q_0(x(z)).
\end{equation}
Therefore, 
\begin{equation}
2 T_{-1}(z) + U_0 (z)
= x'(z) \big\{2 y(z) + q_0(x(z))\big\}
= x'(z) \big\{y(z) - y(\overline{z})\big\}
= x'(z) \Delta (z)
\end{equation}
holds. 
Then, \eqref{eq:Riccati-gen-z-m+1} for $m=-1$ implies that 
the desired relation $T_0(z) = \hat{T}_{0}(z)$ is equivalent to 
\begin{equation}
\label{eq:goal1}
x'(z) \Delta(z)\hat{T}_0
+ \left\{U_1(z) - \frac{x''(z)}{x'(z)}\right\} \hat{T}_{-1}
+ \frac{d}{dz}\hat{T}_{-1} + V_1(z) = 0.
\end{equation}
In order to verify \eqref{eq:goal1}, we first prove the following.

\begin{lem}\label{lem:U1-expression}
The function $U_1(z)$ is expressed as  
\begin{equation}
\label{eq:gual2:equiv}
U_1(z)
=
-\frac{\Delta'(z)}{\Delta(z)} + 
\frac{G_{0, 2}(\overline{z}, z) + G_{0, 2}(z, \overline{z})}{dz}.
\end{equation}
\end{lem}

\begin{proof}[Proof of Lemma \ref{lem:U1-expression}]
It follows from Proposition \ref{prop:G-properties-01} (ii) 
and Proposition \ref{prop:G-properties-02} (i) that 
\begin{align}
 &
\frac{G_{0, 2}(\overline{z}, z) + G_{0, 2}(z, \overline{z})}{dz}
=
\frac{2}{z - \overline{z}}
- \sum_{\beta\in B} \nu_{\beta}\left(\frac{1}{z - \beta} + 
\frac{1}{z - \overline{\beta}}\right)
=
\frac{2}{z - \overline{z}}
- \sum_{\beta\in B} \frac{\nu_{\beta} + \nu_{\overline{\beta}}}{z -\beta}.
\end{align}
Here we have used the fact that $B$ is 
closed under the conjugate map (cf. Remark \ref{rem:MainThm} (iii))
in the last equality. This proves the relation \eqref{eq:gual2:equiv}. 
\end{proof}

Let us prove
\begin{equation}
\label{eq:goal1:equiv}
V_1(z)
= -
x'(z) \Delta(z) \hat{T}_0
- \left\{U_1(z) - \frac{x''(z)}{x'(z)}\right\} \hat{T}_{-1}
- \frac{d}{dz}\hat{T}_{-1},
\end{equation}
which is equivalent to \eqref{eq:goal1}.
By using Proposition \ref{prop:That} and Lemma \ref{lem:U1-expression},
the right-hand side becomes
\begin{align}
 &
-x'(z) \Delta(z) \hat{T}_0
- \left\{U_1(z) - \frac{x''}{x'}\right\} \hat{T}_{-1}
- \frac{d}{dz} \hat{T}_{-1}\\
&\qquad
=
x'(z) \Delta(z) \frac{G_{0, 2}(\overline{z}, z)}{dz}
- \left\{
-\frac{\Delta'(z)}{\Delta(z)} 
+ \frac{G_{0, 2}(\overline{z}, z) + G_{0, 2}(z, \overline{z})}{dz}
- \frac{x''(z)}{x'(z)} \right\} x'(z)y(z)
\notag\\
&\qquad\quad
-\frac{d}{dz} \big(x'(z) y(z)\big)
\notag
\\
&\qquad
=
x'(z) \Delta(z) \frac{G_{0, 2}(\overline{z}, z)}{dz}
+ \frac{\Delta'(z)}{\Delta(z)} x'(z) y(z)
\notag\\
&\qquad\qquad
- \frac{G_{0, 2}(\overline{z}, z) + G_{0, 2}(z, \overline{z})}{dz}
x'(z) y(z)
- x'(z) y'(z).
\notag
\end{align}
Substituting
\[
y(z) = \frac{y(z) - y(\overline{z})}{2} + \frac{y(z) +
 y(\overline{z})}{2}
= \frac{1}{2} \Delta(z) - \frac{1}{2} q_0(x(z))
\]
into the last expression, we find
\begin{align}
&(\text{RHS of \eqref{eq:goal1:equiv}})
\\
&\qquad
=
\frac{1}{2}
x'(z) \Delta(z) \frac{G_{0, 2}(\overline{z}, z) - G_{0, 2}(z, \overline{z})}{dz}
- \frac{1}{2} x'(z) \frac{\Delta'(z)}{\Delta(z)} q_0(x(z))
\notag\\
&\qquad\qquad
+ \frac{1}{2} x'(z)^2 \left. \frac{d q_0}{d x} \right|_{x = x(z)}
+ \frac{1}{2} x'(z) q_0(x(z)) \,
\frac{G_{0, 2}(\overline{z}, z) + G_{0, 2}(z, \overline{z})}{dz}.
\notag
\end{align}
Then, thanks to the equalities \eqref{eq:gual2:equiv} and
\begin{align}
\frac{G_{0, 2}(\overline{z}, z) - G_{0,2}(z, \overline{z})}{dz}
&=
- \sum_{\beta \in B} \frac{\nu_{\beta}}{z - \overline{\beta}}
+ \sum_{\beta \in B} \frac{\nu_{\beta}}{z - \beta}
=
\sum_{\beta \in B} \frac{\nu_{\beta} - \nu_{\overline{\beta}}}{z - \beta},
\end{align}
we conclude that the right-hand side of \eqref{eq:goal1:equiv}
is equal to $V_1(z)$. Thus we have verified \eqref{eq:goal1}, 
and hence, \eqref{eq:main-equation-quantization} holds for $m = 0$.

\medskip
{\bf\boldmath{Proof of $T_1(z) = \hat{T}_{1}(z)$.}}\quad
It follows from Proposition \ref{prop:That} and 
Lemma \ref{lem:U1-expression} that 
$\hat{T}_1(z)$ satisfies 
\begin{align}
 x'(z) \Delta(z) \hat{T}_1(z) 
+ \left\{U_1(z) - \frac{x''(z)}{x'(z)}\right\} \hat{T}_0(z)
+ {\hat{T}_0(z)}^2
+ \frac{d \hat{T}_0}{dz}(z) 
\\ 
\qquad
+ x'(z) \Delta(z) \sum_{\beta  \in B_1}
\frac{\nu_{\beta} \nu_{\overline{\beta}}}{2 C_{\beta}}
\frac{G_{0, 2}(z, \beta) - G_{0, 2}(z, \overline{\beta})}{dz}
= 0. 
\notag
\end{align}
In view of \eqref{eq:Riccati-gen-z-m+1} for $m=0$, the desired relation 
$T_1(z) = \hat{T}_{1}(z)$ is equivalent to 
the following claim:
\begin{lem}\label{lem:V2-expression}
The function $V_2(z)$ is expressed as
\begin{equation}
\label{eq:goal3}
V_2(z) = x'(z) \Delta(z)
\sum_{\beta \in B_1}
\frac{\nu_{\beta} \nu_{\overline{\beta}}}{2 C_{\beta}}
\frac{G_{0,2}(z, \beta) - G_{0,2}(z, \overline{\beta})}{dz}.
\end{equation}
\end{lem}
\begin{proof}[Proof of Lemma \ref{lem:V2-expression}]
It follows from Proposition \ref{prop:G-properties-01} (ii) 
and the relation \eqref{eq:C-beta-bar} that
\begin{align}
\sum_{\beta \in B_1}
\frac{\nu_{\beta} \nu_{\overline{\beta}}}{2 C_{\beta}}
\frac{G_{0,2}(z, \beta) - G_{0,2}(z, \overline{\beta})}{dz}
&=
\sum_{\beta \in B_1}
\frac{\nu_{\beta} \nu_{\overline{\beta}}}{2 C_{\beta}}
\frac{1}{z-\beta}
-
\sum_{\beta \in B_1}
\frac{\nu_{\beta} \nu_{\overline{\beta}}}{2 C_{\beta}}
\frac{1}{z-\overline{\beta}}
\\
&=
\sum_{\beta \in B_1}
\frac{\nu_{\beta} \nu_{\overline{\beta}}}{2 C_{\beta}}
\frac{1}{z-\beta}
+
\sum_{\beta \in B_1}
\frac{\nu_{\beta} \nu_{\overline{\beta}}}{2 C_{\overline{\beta}}}
\frac{1}{z-\overline{\beta}}
\notag\\
&=
\sum_{\beta \in B_1}
\frac{\nu_{\beta} \nu_{\overline{\beta}}}{C_{\beta}}
\frac{1}{z-\beta}.
\notag
\end{align}
In the last equality, we have also used the fact that the set $B_1$ is closed under
the conjugate map (cf. Remark \ref{rem:MainThm} (iii)).
This proves \eqref{eq:goal3}.
\end{proof}
Thus we have verified that 
\eqref{eq:main-equation-quantization} holds for for $m = 1$.

\medskip
{\bf\boldmath{Proof of $T_m(z) = \hat{T}_{m}(z)$ for $m \geq 2$.}}\quad
Proposition \ref{prop:That} and 
Lemma \ref{lem:U1-expression} show that 
\begin{align}
 x'(z) \Delta(z) \hat{T}_{m+1}(z) 
+ \left\{U_1(z) - \frac{x''(z)}{x'(z)}\right\} \hat{T}_0(z)
+ \sum_{j=0}^{m} \hat{T}_{m-j}(z) \hat{T}_{j}(z)
+ \frac{d\hat{T}_m}{dz}
= 0 \quad (m \ge 1)
\notag
\end{align}
holds. This is the same as the equation \eqref{eq:Riccati-gen-z-m+1} 
for $m \ge 1$ satisfied by $T_{m+1}(z)$. 
Thus we have proved the desired relation
\eqref{eq:main-equation-quantization} for all $m \ge 2$. 

\medskip
This completes the proof of Theorem \ref{thm:WKB-Wg,n}.


\section{The Voros coefficient of the Weber equation and the free energy of the Weber curve} 
\label{sec:weber}


In this section, we discuss the Weber curve defiened by
\begin{equation}
\label{Weber_P(x,y)}
P(x, y) := y^2 - \frac{x^2}{4} + \lambda = 0
\quad (\lambda \neq 0),
\end{equation}
and its quantization constructed in the previous section.
This is a model case of our study of quatum curves.
We will establish a relationship between the 
Voros coefficient of quantum curves and the 
free energy of the Weber curve. 
As a corollary, we can obtain an explicit expression of both of 
the Voros coefficients and the free energy, and recover the 
expression of $F_g$ computed in \cite{HZ}.
Some of the results in this subsection are studied in \cite{Takei17}.

\subsection{Weber equation as a quantum curve}
To apply the topological recursion,
we parametrize the Weber curve as
\begin{equation}
\label{Weber_parameterization}
\begin{cases}
\displaystyle
x = x(z) = \sqrt{\lambda} \left(z + \frac{1}{z}\right),\\
\displaystyle
y = y(z) = \frac{\sqrt{\lambda}}{2} \left(z - \frac{1}{z}\right)
\end{cases}
\end{equation}
and regard it as a spectral curve in the sense of Definition \ref{def:spectral-curve}.
Since
\begin{equation}
dx = \sqrt{\lambda}\,\, \frac{z^2 - 1}{z^2} dz,
\end{equation}
ramification points are given by $R = \{+1, -1 \} ~(=R^{\ast})$,
and the conjugate map becomes $\overline{z} = 1/z$.
A straightforward computation gives
\begin{equation} \label{eq:Weber_ydx}
y(\overline{z}) = - y(z),
\quad
\Delta(z) = 2 y(z),
\quad
y(z) dx(z) = \frac{\lambda (z^2-1)^2}{2 z^3} dz.
\end{equation}
The correlation functions and free energies for lower $g$ and $n$ are given be
\begin{align} \label{eq:W03-Weber}
W_{0, 3}(z_1, z_2, z_3)
&= \frac{1}{2 \lambda} \left\{ \frac{1}{(z_1 + 1)^2 (z_2 + 1)^2 (z_3 + 1)^2} \right.\\
&\qquad\qquad
\left. - \frac{1}{(z_1 - 1)^2 (z_2 - 1)^2 (z_3 - 1)^2} \right\} dz_1 \, dz_2 \,dz_3, \notag\\
W_{1, 1}(z) &= - \frac{z^3}{\lambda (z^2 - 1)^4} \, dz ,
\qquad
W_{2, 1}(z) = - \frac{21(z^{11}+ 3 z^9 + z^7)}{(z^2 - 1)^{10} \lambda^3} dz
\end{align}
and
\begin{align} \label{eq:Weber-lower-genus-free-energy}
F_0(\lambda) &= - \frac{3}{4} \lambda^2 + \frac{1}{2} \lambda^2 \log{\lambda},
\quad F_1(\lambda) = - \frac{1}{12} \log{\lambda},
\quad F_2(\lambda) = - \frac{1}{240 \lambda^2}.
\end{align}

Correlation functions and free energies 
have the following homogeneity properties with respect to $\lambda$:
\begin{prop}[{\cite[Theorem 5.3]{EO}}]
\label{prop:weber:homogeneous}
\begin{align}
W_{g, n}(z_1, \cdots, z_n)
&= \frac{1}{\lambda^{2g -2 + n}} \times
 \left(W_{g, n}(z_1, \cdots, z_n) \Big|_{\lambda = 1}\right)
\quad (2g + n \geq 3),\\
F_g (\lambda)
&= \frac{1}{\lambda^{2g -2}} \times F_{g} (1)
\quad (g \geq 2).
\end{align}
\end{prop}

Since
\begin{equation}
Q_0(x) = \frac{x^2}{4} - \lambda, \quad
Q_{\infty}(x) = \frac{1}{x^6} \left( \frac{1}{4} - \lambda x^2 \right),
\end{equation}
we have
\begin{equation}
\Sing (P) = \{\infty\},
\quad
B = \{0, \infty\},
\quad
B_1 = \emptyset.
\end{equation}
(cf.,\,\eqref{eq:Weber_ydx}.)
Therefore we choose
\begin{equation}
\label{Weber_D}
D(z; \underline{\nu}) := [z] - \nu_0 [0] - \nu_{\infty} [\infty]
\qquad
(\nu_0 + \nu_{\infty} = 1, \, \underline{\nu} = (\nu_0, \nu_{\infty}))
\end{equation}
as a divisor for the quantization. 
The quantum curve for \eqref{Weber_P(x,y)} 
constructed by Theorem \ref{thm:WKB-Wg,n} becomes
\begin{equation}
\label{Weber_eq}
\left\{\hbar^2 \frac{d^2}{dx^2} -
\left( \frac{x^2}{4} - \hat{\lambda} \right)
\right\}  \psi = 0,
\end{equation}
where 
\begin{equation} \label{eq:lambda-hat-and-nu}
\hat{\lambda} := \lambda - \frac{\hbar \nu}{2} 
\quad \text{and} \quad 
\nu := \nu_{\infty} - \nu_{0}.
\end{equation}
We call the Schr{\"o}dinger-type equation \eqref{Weber_eq}
quantum Weber curve.

\subsection{The Voros coefficient for the quantum Weber curve}
\label{subsec:Weber-Voros}

A WKB solution of the quantum Weber curve \eqref{Weber_eq} is given by
\begin{equation}
\label{eq:WKB:Weber}
\psi (x, \lambda, \nu;\hbar) := \exp
\left(\int^x S(x, \lambda, \nu; \hbar) dx\right), 
\end{equation}
where 
$S(x, \lambda, \nu; \hbar) = \sum_{m = -1}^{\infty} \hbar^m S_m(x, \lambda, \nu)$ 
is a solution of the Riccati equation \eqref{eq:Riccati-gen}
(or the recursion relations \eqref{eq:Riccati-gen-1}--\eqref{eq:Riccati-gen-3})
associated with \eqref{Weber_eq}. 
First few coefficients are given by 
\begin{equation}
S_{-1}(x, \lambda) = \sqrt{\frac{x^2}{4} - \lambda},
\quad
S_0 (x, \lambda, \nu) = - \frac{x}{2(x^2 - 4\lambda)} +   \frac{\nu}{2\sqrt{x^2-4\lambda}},
\quad
\cdots.
\end{equation}

\noindent
Here a branch of $S_{-1}(x, \lambda)$ remains undetermined.
As is explained in \S\ref{sec:review},
once we fix a branch of $S_{-1}(x, \lambda)$,
all of $S_m(x, \lambda, \nu)$ for $m \geq 0$ are
determined uniquely.

By induction, we can show
\begin{prop}
\label{prop:weber:Sn}
For $m = -1, 0, 1, 2, \cdots$, we have
\begin{align*}
\text{{\rm{(i)}}} &\quad
S_m(x, \lambda, \nu) = O(x^{-2m-1})
\quad
(x \rightarrow \infty).\\
\text{{\rm{(ii)}}} &\quad
S_m(\sqrt{\lambda}x, \lambda, \nu)
= \lambda^{-m-1/2} S_m(x, 1, \nu).
\end{align*}
\end{prop}

The property (i) in Proposition \ref{prop:weber:Sn} 
is nothing but a consequence of the holomorphicity 
of the correlation functions $W_{g,n}$ for $2g+n \ge 3$.

Before introducing the Voros coefficient of 
quantum Weber curve \eqref{Weber_eq}, let us 
make several remarks on properties of Weber curve
to specify a defining path of integration for the Voros coefficient. 

There exist two simple turning points of \eqref{Weber_eq}
at $a_1 = 2\sqrt{\lambda}$ and
$a_2 = -2\sqrt{\lambda}$.
See Fig.~\ref{fig:Weber:stokes} (a) 
for Stokes curves emanating from them. 
By the mapping $x=x(z)$, 
the turning point $a_1$ in the $x$-plane
corresponds to the ramification point $z = 1$ .
Because the map $z \mapsto x = x(z)$ is double covering, 
six Stokes curves emanate from the point $z = 1$ as shown in 
Fig.~\ref{fig:Weber:stokes} (b).
Same is true for $x = a_2$ and $z = -1$.

$x = \infty$ corresponds to 
two points $z=0$ and $\infty$ on $z$-plane. 
These two points also correspond to two different 
behaviors of $S_{-1}(x)$ as $x$ tends to infinity; that is, 
\begin{equation}
\label{eq:Weber_branch}
S_{-1}(x(z), \lambda) \sim 
\begin{cases} 
\dfrac{x(z)}{2} & \text{as $z \to \infty$,} \\[+.5em]
- \dfrac{x(z)}{2} & \text{as $z \to 0$}.
\end{cases}
\end{equation}


\begin{figure}[t]
\begin{center}
${\fbox{\includegraphics[width=.32\textwidth]{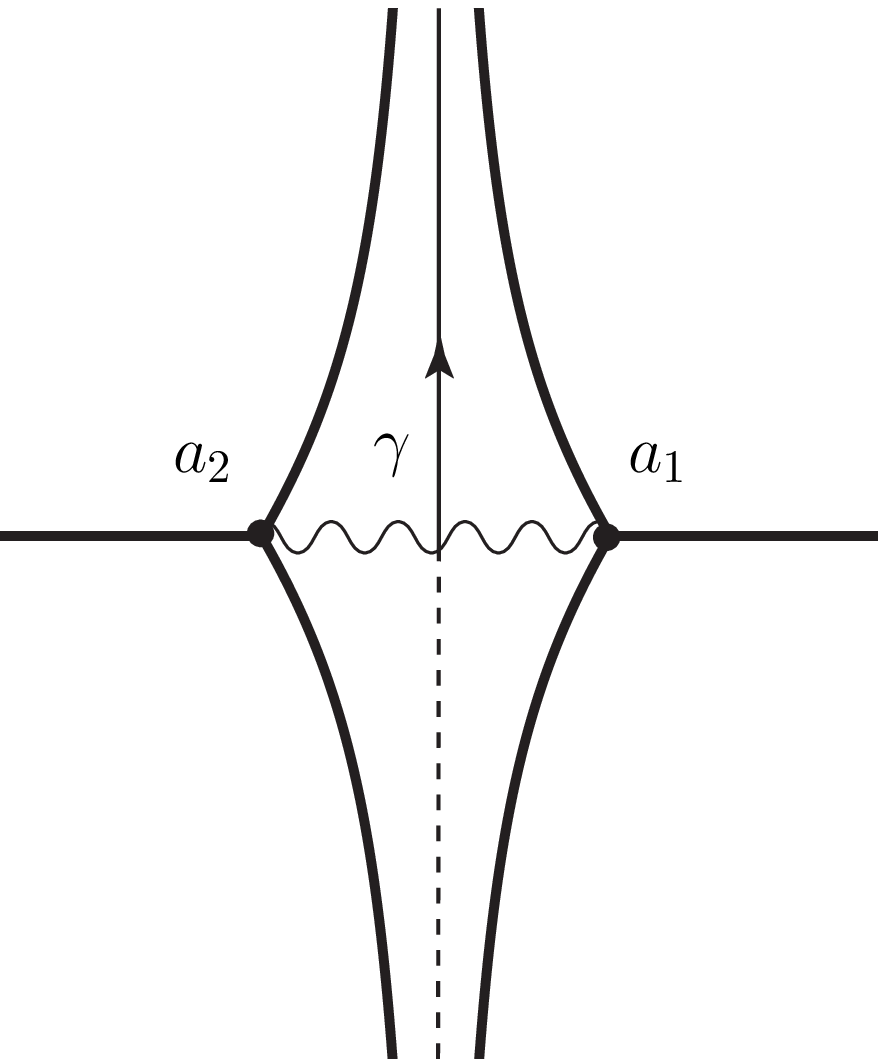}} 
\atop \text{{\normalsize{(a)}}}}$
\quad
${\fbox{\includegraphics[width=.32\textwidth]{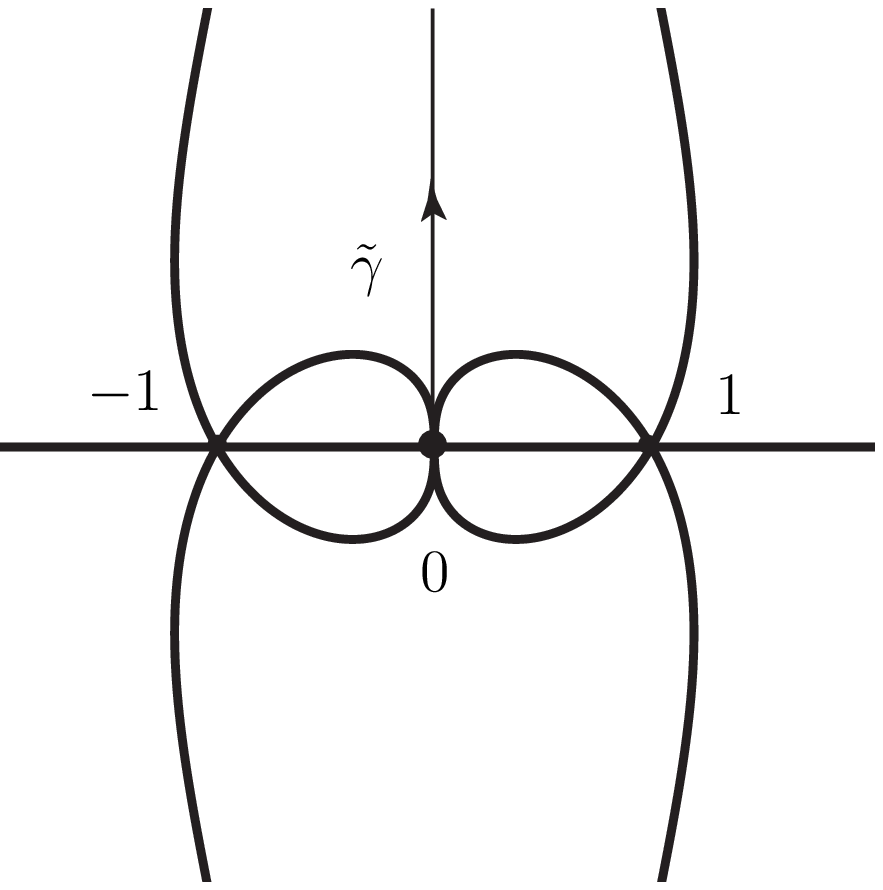}} 
\atop \text{{\normalsize{(b)}}}}$
\end{center}
\caption{(a) Stokes curves of \eqref{Weber_eq} with $\lambda > 0$ 
and the path $\gamma$ on the $x$-plane.
A wiggly line designates a branch cut to define $S_{-1}(x,\lambda)$.
(b) The inverse image of the Stokes curves and $\gamma$ by $x = x(z)$.
}
 \label{fig:Weber:stokes}
\end{figure}

Having the situation in our mind, we define the Voros coefficient
\begin{equation}
\label{Weber_Voros}
V(\lambda, \nu; \hbar)
:= \int_{\gamma} \Bigl( S(x, \lambda, \nu; \hbar) 
- \hbar^{-1} S_{-1}(x,\lambda, \nu) - S_0(x,\lambda, \nu) \Bigr) dx
= \sum_{m = 1}^{\infty} \hbar^m \int_{\gamma} S_m(x, \lambda,\nu) dx.
\end{equation}
Here $\gamma$ is a path from the infinity to itself (on $x$-plane) 
which crosses a branch cut once with $a_1$ on the right-hand side,
as is shown in Fig.~\ref{fig:Weber:stokes} (a).
In the figure, 
a solid (resp. dotted) portion of $\gamma$ means
that it is in a first (resp. second) sheet.
Note also that the path $\gamma$ is the image of 
the path $\tilde{\gamma}$ on $z$-plane from $0$ to $\infty$, 
which is indicated in Fig.~\ref{fig:Weber:stokes} (b).
Thanks to Proposition \ref{prop:weber:Sn} (i), 
this is a well-defined (formal) power series of $\hbar$.

\begin{rem}
Although the Stokes curves are irrelevant in the following discussion, 
those play an important role when we discuss the Borel summability. 
In particular, it is known that the Voros coefficient \eqref{Weber_Voros}
is Borel summable (as a formal series of $\hbar$) if its defining path 
$\gamma$ never intersects with a Stokes curve connecting turning points. 
In addition, the Stokes phenomenon (i.e., jump of Borel sum)
occurring to the (exponential of) Voros coefficient is closely related 
to the cluster transformation. See \cite{IN14} for details.
\end{rem}



\subsection{Relations between the Voros coefficient and the free energy}
\label{subsec:Voros-vs-TR}


In this subsection we formulate our main result which allows us 
to express the Voros coefficient of the quantum Weber curve 
by the free energy of the Weber curve with a parameter shift. 

Let
\begin{equation} \label{eq:total-free-energy-Weber}
F(\lambda; \hbar) := \sum_{g = 0}^{\infty} \hbar^{2g - 2} F_g(\lambda)
\end{equation}
be the free energy of the Weber curve \eqref{Weber_P(x,y)}.
We now claim

\begin{thm}
\label{thm:weber:main(i)}
The Voros coefficient \eqref{Weber_Voros} of the quantum Weber curve 
\eqref{Weber_eq} and the free energy \eqref{eq:total-free-energy-Weber} 
of the Weber curve \eqref{Weber_P(x,y)} are related as follows.
\begin{equation} \label{eq:V-and-F-general}
V(\lambda, \nu; \hbar)
= 
F \left(\hat{\lambda} +\frac{\hbar}{2}; \hbar \right)
- F \left(\hat{\lambda} - \frac{\hbar}{2} ; \hbar \right)
- \frac{1}{\hbar} \frac{\partial F_0}{\partial \lambda}(\lambda) 
+ \frac{ \nu }{2} \frac{\partial^2 F_0}{\partial \lambda^2}(\lambda). 
\end{equation}
(Recall that $\hat{\lambda} = \lambda - (\hbar \nu)/2$ 
and $\nu = \nu_{\infty} - \nu_0$ as we have set in \eqref{eq:lambda-hat-and-nu}.)


\end{thm}

To prove Theorem \ref{thm:weber:main(i)}, 
we need the following identity.

\begin{lem}
\label{lem:Weber_variation}
\begin{equation}
\label{eq:Weber_variation}
\frac{\partial^n}{\partial\lambda^n} F_g
= \int_{\zeta_1 = 0}^{\zeta_1=\infty}\cdots
\int_{\zeta_n = 0}^{\zeta_n=\infty}
W_{g, n}(\zeta_1, \cdots, \zeta_n)\qquad (2g + n \geq 3).
\end{equation}
\end{lem}

\begin{proof}[Proof of Lemma \ref{lem:Weber_variation}]
Because
\begin{equation}
\Omega(z) = \frac{\partial y(z)}{\partial \lambda} \cdot dx(z)
- \frac{\partial x(z)}{\partial \lambda} \cdot dy(z)
= - \frac{dz}{z}
= \int^{\zeta = \infty}_{\zeta = 0} B(z, \zeta)
\end{equation}
holds, Theorem \ref{thm:VariationFormula} 
(and \eqref{eq:iterative-variation})
gives \eqref{eq:Weber_variation}, except for the case $g=0$.
By using the expressions \eqref{eq:W03-Weber} and 
\eqref{eq:Weber-lower-genus-free-energy} of $W_{0,3}$ and $F_0$, 
we can verify \eqref{eq:Weber_variation} holds for for $(g,n) = (0,3)$ directly.
Therefore, thanks to Theorem \ref{thm:VariationFormula},
we can conclude that \eqref{eq:Weber_variation} 
is also valid for $g=0$ and $n \ge 3$. 
This completes the proof.
\end{proof}

\begin{proof}[Proof of Theorem \ref{thm:weber:main(i)}]
By Theorem \ref{thm:WKB-Wg,n},
the Voros coefficient can be rewritten as
\begin{align}
V(\lambda, \nu; \hbar)
&= \sum_{m = 1}^{\infty} \hbar^m \int_0^\infty
\Bigl( S(x(z), \lambda, \nu; \hbar) 
- \hbar^{-1} S_{-1}(x(z), \lambda) 
- S_0(x(z), \lambda, \nu) \Bigr) \frac{dx}{dz} \, dz \\
&= \sum_{m = 1}^{\infty} \hbar^m \int_0^\infty
\left\{ \sum_{\substack{2g + n - 2 = m \\ g \geq 0, \, n \geq 1}}
\frac{1}{n!} \frac{d}{dz} \int_{\zeta_1 \in D(z; \underline{\nu})}
\cdots \int_{\zeta_n \in D(z; \underline{\nu})}
W_{g, n}(\zeta_1, \ldots, \zeta_n) \right\} dz
\notag\\
&= \sum_{m = 1}^{\infty} \hbar^m
\sum_{\substack{2g + n - 2 = m \\ g \geq 0, \, n \geq 1}}
\frac{1}{n!}
\left(
\int_{\zeta_1 \in D(\infty; \underline{\nu})} \cdots \int_{\zeta_n \in D(\infty; \underline{\nu})}
\right.
\notag\\
&\qquad\qquad\qquad\qquad\qquad\qquad
\left.
-
\int_{\zeta_1 \in D(0; \underline{\nu})} \cdots \int_{\zeta_n \in D(0; \underline{\nu})}
\right)
W_{g, n}(\zeta_1, \ldots, \zeta_n).
\notag
\end{align}
Because
\begin{equation}
D(\infty; \underline{\nu}) = \nu_0 ([\infty] - [0])
\quad\text{and}\quad
D(0; \underline{\nu}) = - \nu_{\infty} ([\infty] - [0]),
\end{equation}
we have
\begin{equation}
V(\lambda, \nu; \hbar)
= \sum_{m = 1}^{\infty} \hbar^m \sum_{\substack{2g + n - 2 = m \\ g \geq 0, \, n \geq 1}}
\frac{{\nu_0}^n - (- \nu_{\infty})^n}{n!} \int_0^\infty \cdots \int_0^\infty W_{g, n}(\zeta_1, \ldots, \zeta_n) .
\end{equation}
Now we use Lemma \ref{lem:Weber_variation}:
\begin{align}
V(\lambda, \nu; \hbar)
&= \sum_{m = 1}^{\infty} \hbar^m \sum_{\substack{2g + n - 2 = m \\ g \geq 0, \, n \geq 1}}
\frac{{\nu_0}^n - (-\nu_{\infty})^n}{n!}
\frac{ \partial^n F_g }{ \partial \lambda^n } \\
&= \sum_{n = 1}^{\infty}
\frac{{\nu_0}^n - (-\nu_{\infty})^n}{n!}
\hbar^n
\frac{ \partial^n F(\lambda; \hbar) }{ \partial \lambda^n }
- \frac{\nu_0 - (- \nu_{\infty})}{\hbar}\frac{\partial F_0}{\partial \lambda}
- \frac{{\nu_0}^2 - (-\nu_{\infty})^2}{2!} \frac{\partial^2
 F_0}{\partial \lambda^2}
\notag\\
&= F \left(\lambda + \nu_0 \hbar; \hbar \right)
- F \left(\lambda -\nu_{\infty}\hbar; \hbar \right)
- \frac{\nu_0 + \nu_{\infty}}{\hbar}\frac{\partial F_0}{\partial \lambda}
- \frac{{\nu_0}^2 - {\nu_{\infty}}^2}{2} \frac{\partial^2 F_0}{\partial \lambda^2}.
\notag
\end{align}
Since
$\nu_0 = (1 - \nu)/2$ and
$\nu_{\infty} = (1 + \nu)/2$,
we obtain \thmref{weber:main(i)}.
\end{proof}


\begin{rem} \label{rem:regularization}
In the definition (\ref{Weber_Voros}) of the Voros coefficient, we subtracted the first two terms 
$\hbar^{-1}S_{-1}$ and $S_0$ because these terms are singular at end points of the path $\gamma$. 
However, a regularization procedure of divergent integral (see \cite{Voros-zeta} for example) 
allows us to define the regularized Voros coefficient: 
\begin{equation} 
V_{\rm reg}(\lambda,\nu;\hbar) 
:= 
\hbar^{-1}V_{-1}(\lambda,\nu) 
+ V_0(\lambda,\nu)
+ V(\lambda,\nu;\hbar), 
\end{equation}
where 
$V_{-1}(\lambda,\nu)$ and $V_0(\lambda,\nu)$ are obtained by solving 
\begin{equation} \label{eq:zeta-regularization-equation}
\frac{\partial^2}{\partial \lambda^2} V_{-1} 
= 
\int_{\gamma} \frac{\partial^2}{\partial \lambda^2} S_{-1}(x) \, dx, \quad 
\frac{\partial}{\partial \lambda} V_{0}
 = 
\int_{\gamma} \frac{\partial}{\partial \lambda} S_{0}(x) \, dx.
\end{equation}
Actually, we can verify that $\partial_{\lambda}^2 S_{-1}(x) dx$ and 
$\partial_{\lambda}S_0(x) dx$ are holomorphic at $x=\infty$ although  
$S_{-1}$ and $S_0$ are singular there.
Hence, the equations in \eqref{eq:zeta-regularization-equation} make sense
and are solved explicitly by 
\begin{align}
V_{-1} = \lambda \log{\lambda} - \lambda, \qquad 
V_{0} = - \frac{\nu}{2} \log{\lambda}. 
\end{align}
Actually, we can verify that the regularized integrals are realized by the correction terms 
\begin{equation}
V_{-1} = \frac{\partial F_0}{\partial \lambda}, \qquad
V_{0} = - \frac{\nu}{2} \frac{\partial^2 F_0}{\partial \lambda^2} 
\end{equation}
in the right hand-side of the relation \eqref{eq:V-and-F-general}.
Thus we can conclude that the regularized Voros coefficient satisfies 
\begin{equation} \label{eq:Vreg-and-free-energy}
V_{\rm reg}(\lambda,\nu;\hbar) 
= 
F \left(\hat{\lambda} +\frac{\hbar}{2}; \hbar \right)
- F \left(\hat{\lambda} - \frac{\hbar}{2} ; \hbar \right). 
\end{equation}
\end{rem}


\subsection{Three-term difference equations satisfied by the free energy}
\label{subsec:Three-term_difference-eq.}


In this subsection, we derive the three-term difference equation 
which the generating function of the free energies satisfies. 
The precise statement is formulated as follows.

\begin{thm}
\label{thm:weber:main(ii)}
The free energy \eqref{eq:total-free-energy-Weber} satisfies 
the following difference equation.
\begin{equation}
\label{eq:Weber_free-energy_difference-eq.}
F(\lambda + \hbar; \hbar) - 2 F(\lambda; \hbar) + F(\lambda - \hbar; \hbar)
= \frac{\partial^2 F_0}{\partial \lambda^2}
\; \Big(= \log{\lambda}\Big).
\end{equation}
\end{thm}

To prove Theorem \ref{thm:weber:main(ii)}, 
we need the following identity.

\begin{lem}
\label{lem:Weber_Voros-parameter}
\begin{equation}
\label{eq:Weber_Voros-difference}
V(\lambda, \nu + 2; \hbar)
- V(\lambda, \nu; \hbar) =
- \log \left(1 - \frac{\nu + 1}{2\lambda} \hbar\right).
\end{equation}
\end{lem}

\begin{proof}[Proof of Lemma \ref{lem:Weber_Voros-parameter}]
The following proof of this lemma uses the same idea with that in \cite{Takei08}, 
while the resulting equation \eqref{eq:Weber_Voros-difference}
has a simpler form because
we only consider the difference of Voros coefficients
with respect to $\nu$ (not $\lambda$).

We can easily verify that 
\begin{equation}
\left(\frac{\partial}{\partial x} + \frac{x}{2 \hbar}\right) \psi(\nu) 
= {\rm (const.)} \times \psi(\nu+2)
\end{equation}
holds (cf. \cite{Takei08}). 
By taking the logarithmic derivatives of both sides, we obtain
\begin{equation}
S(x, \nu + 2) - S(x, \nu)
= \frac{d}{dx} \log
\left( \frac{x}{2\hbar} + S(x, \nu)\right).
\end{equation}
By integration, we get
\begin{align}
&\int^x_{\check{x}} \big\{S_0(x, \nu + 2) - S_0(x, \nu)\} dx
+
V^{x, \check{x}} (\nu + 2)
-
V^{x, \check{x}} (\nu)
\\
&\qquad\qquad
=
\log \left(\frac{x}{2\hbar} + S(x, \nu)\right)
-
\log \left(\frac{\check{x}}{2\hbar} + S(\check{x}, \nu)\right),
\notag
\end{align}
where
\begin{equation}
V^{x, \check{x}} (\nu)
=
\int^x_{\check{x}} \big(S(x, \nu ) - \hbar^{-1}S_{-1}(x) - S_0(x, \nu) \big) dx,
\end{equation}
and $x$ (resp., $\check{x}$) is a point on $\gamma$ which will be taken a
limit as $x$ tends to the end point (resp., the initial point) of $\gamma$.
It follows from \eqref{eq:Weber_branch}
and Proposition \ref{prop:weber:Sn} (i) that
\begin{equation}
\frac{x}{2\hbar} + S(x, \nu)
= \frac{x}{\hbar} \left(1 + O(|x|^{-2})\right)
\end{equation}
as $x$ tends to the end point of $\gamma$. Hence we have
\begin{equation}
\log \left(\frac{x}{2\hbar} + S(x, \nu)\right)
= \log \frac{x}{\hbar} + O(|x|^{-2}).
\end{equation}
In a similar manner, it follows from
\begin{equation}
S_{-1}(\check{x}) \sim - \frac{\check{x}}{2} + \frac{\lambda}{\check{x}} + O(|\check{x}|^{-2}),
\quad
S_0 (\check{x}) \sim - \frac{\nu + 1}{2\check{x}} + O(|\check{x}|^{-2})
\end{equation}
(care is needed for the branch, cf. \eqref{eq:Weber_branch})
and Proposition \ref{prop:weber:Sn} (i) that
\begin{equation}
\log \left(\frac{\check{x}}{2\hbar} + S(\check{x}, \nu)\right)
=
\log \left(\lambda - \frac{\nu + 1}{2}\hbar \right) 
- \log (\hbar \check{x}) + O(|\check{x}|^{-1})
\end{equation}
as $\check{x}$ tends to the end point of $\gamma$. Hence
\begin{align}
&\int^x_{\check{x}} \big\{S_0(x, \nu + 2) - S_0(x, \nu)\} dx
+
V^{x, \check{x}} (\nu + 2)
-
V^{x, \check{x}} (\nu)
\\
&\qquad\qquad
=
\log (x \check{x}) - \log \left(\lambda - \frac{\nu + 1}{2} \hbar 
 \right) + O(|x|^{-2}) + O(|\check{x}|^{-1}).
\notag
\end{align}
Because $S_0(x, \nu + 2)$ and $S_0(x, \nu)$ do not depend on $\hbar$,
and $V^{x, \check{x}} (\nu + 2)$ and $V^{x, \check{x}} (\nu)$ do not
contain constant term with respect to $\hbar$, we finally
obtain\footnote{%
Note that \eqref{eq:Weber_Voros-parameter:tmp1} also follows from
the indefinite integral
$$
\int^x \big\{S_0(x, \nu + 2) - S_0(x, \nu)\} dx
= \int^x \frac{dx}{\sqrt{x^2 - 4\lambda}}
= \log \left(x + \sqrt{x^2 - 4\lambda} \right).
$$
}
\begin{align}
\label{eq:Weber_Voros-parameter:tmp1}
\int^x_{\check{x}} \big\{S_0(x, \nu + 2) - S_0(x, \nu)\} dx
&= \log (x \check{x}) - \log \lambda + O(|x|^{-2}) + O(|\check{x}|^{-1}),
\\
\label{eq:Weber_Voros-parameter:tmp12}
V^{x, \check{x}} (\nu + 2) - V^{x, \check{x}} (\nu)
&= - \log \left(1 -  \frac{\nu + 1}{2\lambda} \hbar \right)
+ O(|x|^{-2}) + O(|\check{x}|^{-1}).
\end{align}
The last equality \eqref{eq:Weber_Voros-parameter:tmp12}
proves Lemma \ref{lem:Weber_Voros-parameter} 
after taking the limit $x, \hat{x} \to \infty$ 
along $\gamma$. 
\end{proof}

\begin{proof}[Proof of Theorem \ref{thm:weber:main(ii)}]
By Lemma \ref{lem:Weber_Voros-parameter}, we have
\begin{equation}
\label{thm:weber:main:tmpeq}
V(\lambda, \nu; \hbar)\big|_{\nu = 1}
=
V(\lambda, \nu; \hbar)\big|_{\nu = -1}.
\end{equation}
It follows from \thmref{weber:main(i)} that
\begin{eqnarray}
V(\lambda, \nu; \hbar)\big|_{\nu = 1}
& = &
F \left(\lambda ; \hbar \right)
- F \left(\lambda - \hbar; \hbar \right)
- \frac{\partial F_0}{\partial \lambda} \hbar^{-1}
+ \frac{ 1 }{2} \frac{\partial^2 F_0}{\partial \lambda^2}, \\
V(\lambda, \nu; \hbar)\big|_{\nu = -1}
& = &
F \left(\lambda + \hbar; \hbar \right)
- F \left(\lambda; \hbar \right)
- \frac{\partial F_0}{\partial \lambda} \hbar^{-1}
- \frac{ 1 }{2} \frac{\partial^2 F_0}{\partial \lambda^2}.
\end{eqnarray}
By substituting these two relations
into \eqref{thm:weber:main:tmpeq}, we obtain \thmref{weber:main(ii)}.
\end{proof}


\subsection{The concrete form of the free energy}
\label{subsec:FreeEnergy}


As a corollary of the main result, we can find explicit formulas 
for the coefficients of the free energy and the Voros coefficient. 
In this subsection, we derive the explicit form of the free energy. 
For the explicit form of the Voros coefficient, see the next subsection. 

\begin{thm}[cf.\,{\cite{HZ}; see also \cite{IM}}]
\label{cor:weber:FreeEnergy}
The $g$-th free energy of the Weber curve \eqref{Weber_P(x,y)}
is given explicitly as follows:
\begin{equation} \label{eq:Weber-free-energy-explicit}
F_g(\lambda) = \dfrac{B_{2g}}{2g(2g - 2)} \dfrac{1}{\lambda^{2g-2}} \quad (g \geq 2),
\end{equation}
where $\{B_n\}_{n \geq 0}$ designates the Bernoulli number defined 
through the generating function as
\begin{equation}
\label{def:Bernoulli}
\frac{w}{e^w - 1} = \sum_{n = 0}^{\infty} B_n \frac{w^n}{n!}.
\end{equation}
\end{thm}

\begin{proof}
By using a shift operator 
(or an infinite order differential operator)
$e^{\hbar\partial_{\lambda}}$,
the equation (\ref{eq:Weber_free-energy_difference-eq.}) 
in Theorem \ref{thm:weber:main(ii)} becomes
\begin{equation}
\label{prop:difference-eq:sol:tmp:1}
e^{-\hbar\partial_{\lambda}} (e^{\hbar\partial_{\lambda}} - 1)^2
F = \log \lambda.
\end{equation}
It follows from
\begin{equation}
e^{-w} (e^w - 1)^2
\left\{
\frac{1}{w^2}
- \sum_{n = 0}^{\infty} \frac{B_{n + 2}}{\, n + 2 \,} \frac{\, w^n \,}{\, n! \,}
\right\}
= 1
\end{equation} 
(which follows from the definition \eqref{def:Bernoulli} of the Bernoulli numbers) 
that
\begin{equation}
e^{-\hbar\partial_{\lambda}} (e^{\hbar\partial_{\lambda}} - 1)^2
\left\{
(\hbar\partial_{\lambda})^{-2}
- \sum_{n = 0}^{\infty} \frac{B_{n + 2}}{\, n + 2 \,} 
\frac{\, (\hbar\partial_{\lambda})^n \,}{\, n! \,}
\right\}
= {\rm{id}}.
\end{equation}
Hence we find that
\begin{align}
\label{weber:sol:FreeEnergy}
\hat{F}(\lambda;\hbar)
&:=
\left\{
(\hbar\partial_{\lambda})^{-2}
- \sum_{n = 0}^{\infty} \frac{B_{n + 2}}{\, n + 2 \,} 
\frac{\, (\hbar\partial_{\lambda})^n \,}{\, n! \,}
\right\}
\log \lambda \\
&= 
(\hbar \partial_{\lambda})^{-2} \log \lambda
- \frac{1}{12} \log \lambda
+ \sum_{g = 2}^{\infty} \frac{B_{2g}}{2g(2g-2)} \frac{\hbar^{2g - 2}}{\lambda^{2g-2}} 
\notag
\end{align}
is a solution of \eqref{eq:Weber_free-energy_difference-eq.}.
Here we note that, 
\begin{equation}
\frac{\partial^2 F_0}{\partial \lambda^2} = \log \lambda
\end{equation}
holds by \eqref{eq:Weber-lower-genus-free-energy}, 
and hence, we may choose $\hbar^{-2} F_0$ 
as the top term in \eqref{weber:sol:FreeEnergy}.

Since $F$ and $\hat{F}$ satisfies the same difference equation 
\eqref{eq:Weber_free-energy_difference-eq.}, 
their difference $G := F - \hat{F} = 
\sum_{g=2}^{\infty} \hbar^{2g-2} G_{g}(\lambda)$ satisfies 
\begin{equation}
G(\lambda + \hbar; \hbar) - 2G(\lambda; \hbar) + G(\lambda - \hbar; \hbar) = 0.
\end{equation}
This relation implies that, the each coefficient $G_{g}(\lambda)$ 
of $G$ must satisfies $\partial_\lambda^2 G_{g} = 0$.
Therefore, each term of $G$ must be a linear in $\lambda$. 
However, due to the homogeneity (Proposition \ref{prop:weber:homogeneous}), 
$F_g - \hat{F}_g$ must be zero for all $g$. 
This shows the desired equality \eqref{eq:Weber-free-energy-explicit}.
\end{proof}


\subsection{The concrete form of the Voros coefficient}
\label{subsec:voros}


Here we derive the concrete form of the Voros coefficient. 
Specifically, we prove the following corollary.
(Note that the special case $\nu=0$ of the claim has 
already obtained in \cite{Voros83, SS, Takei08}.) 

\begin{thm}
\label{cor:weber:voros}
The Voros coefficient of the quantum Weber curve \eqref{Weber_eq}
is given explicitly as follows:
\begin{equation}
V(\lambda, \nu; \hbar)
= \sum_{m = 1}^{\infty} \frac{B_{m + 1}\big( (\nu+1)/2 \big)}{m (m + 1)} 
\left( \frac{\hbar}{\lambda} \right)^{m}.
\end{equation}
Here $B_m(t)$ is the Bernoulli polynomial defined through the generating function as
\begin{equation}
\label{def:BernoulliPoly}
\frac{w e^{X w}}{e^w - 1} = \sum_{m = 0}^{\infty} B_m(X) \frac{w^m}{m!}.
\end{equation}
\end{thm}

\begin{proof}
The relation \eqref{eq:Vreg-and-free-energy} 
between the regularized Voros coefficient 
and the free energy can be written as
\begin{equation}
V_{\rm reg}(\lambda, \nu; \hbar) 
= e^{- \nu\hbar\partial_{\lambda}/2}
\Big(
e^{\hbar\partial_{\lambda}/2} - e^{-\hbar\partial_{\lambda}/2}\Big) F(\lambda; \hbar)
\end{equation}
by the shift operators. 
Using the three term relation \eqref{eq:Weber_free-energy_difference-eq.} of $F$, 
we have
\begin{equation} \label{eq:diffrence-eq-for-Vreg}
e^{\nu\hbar\partial_{\lambda}/2}
\Big(
e^{\hbar\partial_{\lambda}/2} - e^{-\hbar\partial_{\lambda}/2}\Big)
V_{\rm reg}(\lambda, \nu; \hbar)
= \Big(
e^{\hbar\partial_{\lambda}/2} - e^{-\hbar\partial_{\lambda}/2}\Big)^2 F(\lambda; \hbar) 
= \log \lambda.
\end{equation}

Let us invert the shift operator 
$e^{\nu\hbar\partial_{\lambda}/2}
\left(e^{\hbar\partial_{\lambda}/2} - e^{-\hbar\partial_{\lambda}/2}\right)$
(or solving the difference equation) 
to obtain an expression of $V_{\rm reg}$. 
For the purpose, we use a similar technique used in the previous subsection. 
Namely, it follows from 
\begin{equation}
e^{- (X-\frac{1}{2}) w} (e^{w/2}-e^{-w/2}) \left(\frac{1}{w} 
+ \sum_{m=0}^{\infty} \frac{B_{m+1}(X)}{m+1} \frac{w^m}{m!} \right)  = 1
\end{equation}
(cf.\,\eqref{def:BernoulliPoly}) that 
\begin{equation}
e^{- (X-\frac{1}{2}) \hbar \partial_\lambda} 
(e^{\hbar \partial_\lambda/2}-e^{-\hbar \partial_\lambda/2}) 
\left( (\hbar \partial_\lambda)^{-1} 
+ \sum_{m=0}^{\infty} \frac{B_{m+1}(X)}{m+1} 
\frac{(\hbar \partial_\lambda)^m}{m!} \right)  = {\rm id}
\end{equation}
holds. The last equality with $X = (1-\nu)/2$ shows that the formal series
\begin{align} \label{eq:expression-Vreg}
\hat{V}_{\rm reg} 
& = 
\hbar^{-1} \bigl( \lambda\log \lambda - \lambda \bigr)
- \frac{\nu}{2} \log \lambda 
+ \sum_{m=1}^{\infty} \frac{B_{m+1}\bigl( (1-\nu)/2 \bigr)}{m+1} 
\frac{(\hbar \partial_\lambda)^m \log \lambda}{m!} \\
\notag 
& = \hbar^{-1} V_{-1} + V_0 + 
\sum_{m=1}^{\infty} \frac{B_{m+1}\bigl( (1+\nu)/2 \bigr)}{m(m+1)} 
\left(\frac{\hbar}{\lambda}\right)^m
\end{align}
satisfies the difference equation \eqref{eq:diffrence-eq-for-Vreg}.
Here we used $B_1(X) = X - 1/2$ and the equality 
$B_m(X) = (-1)^{m}B_m(1-X)$.

Note that the expression \eqref{eq:expression-Vreg} 
may differ from $V_{\rm reg}$ by a formal power series
of $\hbar$ with constant coefficients in an additive manner 
due to the ambiguity in solutions of difference equations. 
However, such difference are in fact shown to be zero
by a similar argument as that given in the end of the proof
of Theorem \ref{cor:weber:FreeEnergy} by noting the homogeneity property
\begin{equation}
V_m(\lambda, \nu)
:= \int_{\gamma}S_m(x, \lambda, \nu) dx
= \lambda^{-m} V_m(1, \nu),
\end{equation}
which follows from Proposition \ref{prop:weber:Sn} (ii).
Thus we have $V_{\rm reg} = \hat{V}_{\rm reg}$, 
which proves Theorem \ref{cor:weber:voros}.
\end{proof}

\appendix


\section{Meromorphic multidifferentials}
\label{sec:multidifferential}

The correlation function $W_{g, n}(z_1, z_2, \cdots, z_n)$
is a meromorphic multidifferential\,,
i.e., a meromorphic section of the line bundle
$\pi_1^{*}(T^*\mathbb{P}^1) \otimes \pi_2^{*}(T^*\mathbb{P}^1) \otimes
\cdots \otimes \pi_n^{*}(T^*\mathbb{P}^1)$ on $(\mathbb{P}^1)^n$,
where $\pi_j: (\mathbb{P}^1)^n \rightarrow \mathbb{P}^1$ denotes
the $j$-th projection (\cite{DN16}).
Thus a multidifferential is a meromorphic differential in $\mathbb{P}^1$
for each variable. If all of the residue with respect to each variable
vanish, then we call it a multidifferential of the second kind.
We summarize here some notations on multidifferential
which we  use in this paper.
\medskip

{\bf{Local coordinate representation.}}\quad
In a local coordinate, we express
a meromorphic multidifferential $\Omega$ on $(\mathbb{P}^1)^n$ as
\begin{equation}
\Omega = \Omega (z_1, z_2, \cdots, z_n) = f(z_1, z_2, \cdots, z_n) \, 
dz_1 dz_2 \cdots dz_n,
\end{equation}
where $f$ is a meromorphic function 
(we omit the tensor product $\otimes$ in its expression).
If $\Omega$ is symmetric under the permutation of variables, i.e.,
\begin{equation}
\Omega (z_1, \cdots, z_j , \cdots , z_k , \cdots z_n)
= \Omega (z_1, \cdots, z_k , \cdots , z_j , \cdots z_n)
\end{equation}
for any $j, k \in \{1, 2, \cdots, n\}$,
$\Omega$ is said to be a symmetric multidifferential.
\medskip

{\bf{Integration.}}\quad
Its integral with respect to $j$-th variable is denoted by
\begin{equation}
\int_{z_j = a}^{z_j = b}
\Omega(z_1, z_2, \cdots, z_n)
:=
\left(\int_a^b f(z_1, z_2, \cdots, z_n) \, dz_j \right)
dz_1 \cdots dz_{j-1} dz_{j + 1} \cdots dz_n
\end{equation}
or
\begin{equation}
\int_{z_j \in \gamma}
\Omega(z_1, z_2, \cdots, z_n)
:=
\left(\int_{\gamma} f(z_1, z_2, \cdots, z_n) \, dz_j \right)
dz_1 \cdots dz_{j-1} dz_{j + 1} \cdots dz_n
\end{equation}
for an integration path $\gamma$ in $\mathbb{P}^1$.
If $\Omega$ is symmetric under the permutation of the variables,
we write a multiple integral with a same integration contour $\gamma$
like
\begin{equation}
\label{eq:multidiff-tmp1}
\underbrace{\int_{\gamma} \cdots \int_{\gamma}}_{\text{$n$-th}}
 \Omega(z_1, \cdots, z_n)
:=
\int_{\gamma} dz_1 \cdots \int_{\gamma} dz_n \, f(z_1, \cdots, z_n).
\end{equation}
Further we sometimes drop off the word ``$n$-th'' if it is clear
from the context.
\medskip

{\bf{Integration with a divisor.}}\quad
Let $\omega$ be a meromorphic differential on some domain $U$ in
$\mathbb{P}^1$ of the second kind.
Following \cite{BE}, for a divisor
$D(z; \underline{\nu}) = [z] - \sum_{k = 1}^m \nu_k [p_k]$,
where
$z, p_1, p_2, \cdots, p_m \in \mathbb{P}^1 \setminus \{\text{poles of $\omega$}\}$,
and $\underline{\nu} = (\nu_k)_{k=1}^{m}$ being a tuple of 
complex numbers satisfying $\sum_{k = 1}^m \nu_k = 1$, 
we define the integral of $\omega$ with $D(z; \underline{\nu})$ by
\begin{equation}
\int_{D(z; \nu)} \omega := \sum_{k = 1}^m \nu_k \int^z_{p_k} \omega.
\end{equation}
Because $\omega$ is of the second kind, this integral does not
depend on the choice of paths in $U$ from $p_j$ to $z$ ($j = 1, 2, \cdots, m$).
For a multidifferential $\Omega = \Omega (z_1, z_2, \cdots, z_n)$ of the second kind,
we define
\begin{equation}
\int_{z_j \in D(z; \underline{\nu})} 
\Omega(z_1, \cdots, z_j, \cdots z_n)
:=
\sum_{k = 1}^m \nu_k \int^{z_j = z}_{z_j = p_k}
\Omega(z_1, \cdots, z_j, \cdots, z_n).
\end{equation}
As in \eqref{eq:multidiff-tmp1}, we write the multiple integral as
\begin{equation}
\underbrace{
\int_{D(z; \underline{\nu})} 
\cdots \int_{D(z; \underline{\nu})}
}_{\text{$n$-th}}
 \Omega(z_1, \cdots, z_n)
:=
\int_{z_1 \in D(z; \underline{\nu})} dz_1 \cdots 
\int_{z_n \in D(z; \underline{\nu})} dz_n \, f(z_1, \cdots, z_n)
\end{equation}
for a symmetric meromorphic multidifferential of the second kind.
\medskip

{\bf{Pull-back.}}\quad
Finally, for a holomorphic map $\phi$ from some domain in $\mathbb{P}^1$
to $\mathbb{P}^1$,
we write the pullback of $\Omega$ by $\phi$ with respect to the $j$-th variable
as
$$
\Omega(z_1, \cdots, \phi (z_j), \cdots, z_n)
:= f(z_1, \cdots, \phi(z_j), \cdots, z_n)
dz_1 \cdots d\phi(z_j) \cdots dz_n.
$$
We frequently use this expression mainly when  $\phi$ is
conjugate map defined near a ramification point.


\section{Ineffectiveness of ramification points}
\label{sec:miscTR}

Following \cite{BE},
we define $R$ as a set of ramification points of $x(z)$,
not as a set of zeros of $dx(z)$ as in \cite{EO}. 
However, this modification does not cause difference 
when the ramification points are ineffective 
(in the sense of Definition \ref{def:effective-ramification}). 
Here we give a criterion for the ineffectiveness of ramification points 
(cf. Proposition \ref{prop:ineffective}).

\begin{prop} \label{prop:ineffective-ramification}
For a ramification point $r$, the followings are equivalent: 
\begin{itemize}
\item[\rm (a)]
the correlation function $W_{g,n}(z_1, \cdots, z_n)$  
with $(g,n) \ne (1,0)$ is holomorphic at $z_i = r$ for each $i = 1,\cdots, n$ 
(i.e., $r$ is an ineffective ramification point).
\item[\rm (b)]
The differential $(y(z) - y(\overline{z})) dx(z)$ has a pole at $r$.
\end{itemize}
\end{prop}

\begin{proof}
First let us give a remark on the pole order of correlation functions 
at a ramification point satisfying the above condition. 
\begin{lem} \label{lem:even-is-not-ramification}
If $(y(z) - y(\overline{z})) dx(z)$ has a pole at a ramification point $r$, 
then the pole order of $(y(z) - y(\overline{z})) dx(z)$ at $z=r$ 
is greater than or equal to two. 
\end{lem}
\begin{proof}[Proof of Lemma \ref{lem:even-is-not-ramification}]
It is enough to prove that $(y(z) - y(\overline{z})) dx(z)$ never has 
a simple pole at ramification point. 
In other words,  it suffices to prove that, there is no $r \in R$ satisfying $\rho(x(r); P) = -2$
(cf. Proposition \ref{prop:order-of-ydx}). 

Suppose for contradiction that a point $r \in R$ satisfies $\rho(x(r); P) = -2$. 

We also assume that $x(r) \ne \infty$ for simplicity. 
(The case $x(r) = \infty$ can be treated by a similar way.) 
Then, the function $Q_0(x)$ defined by \eqref{eq:Q0-in-section-3} has an expression 
\[
Q_0(x) = \frac{c_0}{(x - x(r))^2} (1 + f(x)),
\]
where $c_0$ is a nonzero constant, and 
$f(x)$ is a rational function of $x$ which vanishes at $x=x(r)$. 
Taking a square root with an appropriate branch, we obtain
\[
y(z) - y(\overline{z}) = \frac{2\sqrt{c_0}}{x(z) - x(r)} (1 + f(x(z)))^{1/2}.
\] 
Since the left hand side is anti-invariant under the involution $z \mapsto \overline{z}$, 
the above equality and the relation $x(z) = x(\overline{z})$ imply 
\[
\frac{1}{x(z) - x(r)} = - \frac{1}{x(z) - x(r)}
\]
which leads a contradiction. This proves that no $r \in R$ satisfying $\rho(x(r); P) = -2$. 
\end{proof}
\begin{rem}
By a similar argument presented in the proof of 
Lemma \ref{lem:even-is-not-ramification}, 
we can also show that, there is no $r \in R$ satisfying 
$\rho(x(r); P) = -2m$ for some $m \ge 1$.
\end{rem}

Now let us prove Proposition \ref{prop:ineffective-ramification}. 

Let us assume (b), that is, $(y(z) - y(\overline{z})) dx(z)$ has a pole at $r$, 
and look at the behavior of $W_{0,3}(z_0,z_1,z_2)$ and 
$W_{1,1}(z_0)$ when $z_0$ approaches to $r$. 
By deforming the residue contour around $r$, we can decompose 
the contribution of residue at $z=r$ to $W_{0,3}(z_0,z_1,z_2)$ as 
\begin{multline} \label{eq:deform-W03}
\Res_{z=r} K_r(z_0,z) \bigl( B(z,z_1)B(\overline{z},z_2) 
+ B(\overline{z},z_1)B(z,z_2) \bigr) \\
= \frac{1}{2 \pi i} \left( \oint_{z \in C_{r,z_0, \overline{z_0}}} 
- \oint_{z \in C_{z_0}} -  \oint_{z \in C_{\overline{z_0}}}   \right)
K_r(z_0,z) \bigl( B(z,z_1)B(\overline{z},z_2) + B(\overline{z},z_1)B(z,z_2) \bigr).
\end{multline}
Here $C_{r,z_0, \overline{z_0}}$ is a contour satisfying 
\[
\{ \text{the domain bounded by $C_{r,z_0, \overline{z_0}}$} \} \cap 
R \cap \{z_0, \overline{z_0}, \cdots, z_n , \overline{z_n}  \} 
= \{r, z_0, \overline{z_0} \},
\]
and the other contours $C_{z_0}$ and $C_{\overline{z_0}}$ are defined similarly.
Then, the first integral in the right hand-side of \eqref{eq:deform-W03} 
is holomorphic at $z_0 = r$, while the other two integrals are evaluated as 
\begin{multline} \label{eq:W03-singular}
- \frac{1}{2 \pi i} \left(\oint_{z \in C_{z_0}} + \oint_{z \in C_{\overline{z_0}}} \right) 
K_r(z_0,z) \bigl( B(z,z_1)B(\overline{z},z_2) + B(\overline{z},z_1)B(z,z_2) \bigr) \\
 = \frac{1}{(y(z_0) - y(\overline{z_0})) dx(z_0)} 
\bigl( B(z_0,z_1)B(\overline{z_0},z_2) + B(\overline{z_0},z_1)B(z_0,z_2) \bigr). 
\end{multline}
Therefore $W_{0,3}(z_0,z_1,z_2)$ is holomorphic at $z_0 = r$ 
(and hence, it is holomorphic at $z_i = r$ for $i=1,2$ as well) 
under the assumption on the pole property of $( y(z) - y(\overline{z}) )dx(z)$ at $r$. 

On the other hand, the behavior of $W_{1,1}(z_0)$ when $z_0$ approaches to $r$ 
is described in a similar manner: 
\begin{equation} \label{eq:deform-W11}
\Res_{z=r} K_r(z_0,z) B(z, \overline{z}) =
\frac{1}{2 \pi i} \oint_{z \in C_{r,z_0, \overline{z_0}}} K_r(z_0,z) B(z, \overline{z})
+ \frac{1}{(y(z_0) - y(\overline{z_0})) dx(z_0)}  B(z_0, \overline{z_0}).
\end{equation}
Then we can conclude that $W_{1,1}(z_0)$ is holomorphic at $z_0 = r$
because $(y(z_0) - y(\overline{z_0})) dx(z_0)$ has a double or higher order pole at $r$
(cf. Lemma \ref{lem:even-is-not-ramification})
while $B(z_0,\overline{z_0})$ has a double pole there.

Using the induction on $2g-2+n$, we can conclude that the 
correlation functions $W_{g,n}$ are holomorphic at $r$ 
with respect to each variable $z_i$, as follows. 
For general $(g,n)$, the contribution of the residue at $z=r$ 
to $W_{g,n+1}(z_0, z_1, \cdots, z_n)$ is given by the following form
\begin{multline} \label{eq:deform-Wgn}
\Res_{z=r} K_r(z_0,z) F_{g,n}(z,\overline{z},z_1,\cdots,z_n) \\
= \frac{1}{2 \pi i} \left( \oint_{z \in C_{r,z_0, \overline{z_0}}} 
- \oint_{z \in C_{z_0}} -  \oint_{z \in C_{\overline{z_0}}}   \right)
K_r(z_0,z) F_{g,n}(z,\overline{z},z_1,\cdots,z_n).
\end{multline}
Although we omit an explicit expression of $F_{g,n}(z,\overline{z},z_1,\cdots,z_n)$ 
(which can be read off from \eqref{eq:TR}), 
we know it has at most double pole at $z=r$ under the induction hypothesis. 
Then, by the similar argument for $W_{1,1}(z_0)$ presented above, 
we can verify that the right hand-side of \eqref{eq:deform-Wgn} is holomorphic at $z_0 = r$. 
Thus we have verified that $r$ is ineffective. 

Conversely, let us assume (a). 
Then, it follows from the definition of ineffectiveness that the correlation functions 
$W_{0,3}(z_0,z_1,z_2)$ must be holomorphic at $z_0 = r$. 
Then, in view of \eqref{eq:W03-singular}, we can conclude that the differential 
$(y(z_0) - y(\overline{z_0})) dx(z_0)$ must have a pole at $z_0 = r$ 
(otherwise the right-hand side of \eqref{eq:W03-singular} never becomes 
holomorphic at $z_0 = r$). 

Thus we have proved the equivalence between the conditions (a) and (b). 
This also completes the proof of Proposition \ref{prop:ineffective} (i).
\end{proof}

The remaining task for a proof of Proposition \ref{prop:ineffective} is to show 
\begin{prop} \label{prop:ineffective-ramification-2}
If $r$ is an ineffective ramification point, 
then the residue at $r$ in \eqref{eq:TR} becomes zero.
\end{prop}

\begin{proof}
Since $\int^{\zeta=z}_{\zeta=\overline{z}}B(z_0,\zeta)$
is holomorphic and vanishes at $z=r$, we can verify that 
$K_r(z_0,z)$ has a double (or more higher order) zero at $z = r$ 
if $r$ is ineffective (cf. Proposition \ref{prop:ineffective-ramification}).
Therefore, $K_r(z_0,z) F_{g,n}(z,\overline{z},z_1,\cdots,z_n)$ 
in \eqref{eq:deform-Wgn} is holomorphic and has no residue at $z = r$. 
This completes the proof.
\end{proof}


\end{document}